\documentclass[11pt]{amsart}
\usepackage{color}
\usepackage{amsmath,amscd,amssymb,amsfonts,amsthm}
\usepackage{tcolorbox}
\usepackage{tikz-cd}
\usepackage{tensor}
\usepackage{enumitem}
\usepackage{mathtools}
\usepackage[all]{xy}

\newtheorem{theorem}{Theorem}[section]
\newtheorem{corollary}[theorem]{Corollary}
\newtheorem{proposition}[theorem]{Proposition}

\newtheorem{lemma}[theorem]{Lemma}
\theoremstyle{definition}    
\newtheorem{definition}[theorem]{Definition}
\theoremstyle{remark}

\newtheorem{remark}[theorem]{Remark}

\newtheorem{example}[theorem]{Example}

\newcommand{\sC}{\mathsf{C}}

\renewcommand{\L}{\mathcal{L}}
\newcommand{\F}{\mathcal{F}}

\newcommand{\sD}{\mathsf{D}}
\newcommand{\sX}{\mathsf{X}}
\newcommand{\sY}{\mathsf{Y}}
\newcommand{\sj}{\mathsf{j}}
\newcommand{\sq}{\mathsf{q}}

\newcommand{\R}{\mathbb{R}}
\newcommand{\C}{\mathbb{C}}

\renewcommand{\a}{\mathsf{a}}

\newcommand{\id}{\on{id}}
\newcommand\lie[1]{\mathfrak{#1}}
\newcommand{\K}{\mathsf{k}}
\newcommand{\h}{\lie{k}}
\newcommand{\g}{\lie{g}}

\renewcommand{\t}{\lie{t}}
\renewcommand{\k}{\lie{k}}

\newcommand{\gln}{\lie{gl}_n}
\newcommand{\on}{\operatorname}

\newcommand{\Ad}{ \on{Ad} }

\renewcommand{\ker}{ \on{ker}}
\newcommand{\rk}{ \on{rank}}

\newcommand{\Bis}{\on{Bis}}

\newcommand{\sz}{\mathsf{s}}
\newcommand{\uz}{\mathsf{u}}
\newcommand{\tz}{\mathsf{t}}
\newcommand{\mz}{\mathsf{m}}
\newcommand{\iz}{\mathsf{inv}}

\newcommand\qu{/\kern-.7ex/} 

\newcommand{\lra}{\longrightarrow}

\renewcommand{\d}{{\mathsf{d}}}

\newcommand{\rra}{\rightrightarrows}

\renewcommand{\subset}{\subseteq}
\renewcommand{\sp}{\mathsf{p}}
\newcommand{\Ra}{\Rightarrow}
\newcommand{\si}{{\mathsf{i}}}

\begin{document}
\title{On Relative Cohomology for Lie Groupoids and Lie Algebroids}

\author{Mar\'ia Amelia Salazar}\address{Escuela de Ciencias Aplicadas e Ingeniería,  Universidad EAFIT (Colombia)}\email{masalazarp@eafit.edu.co}

\begin{abstract} 
Motivated by our attempt to understand characteristic classes of Lie groupoids and geometric structures, we are brought back to the fundamentals of the cohomology theories of Lie groupoids and algebroids. One element that was missing in the literature was the notion of {\em relative cohomology} in this setting. The main aim of this paper is to develop the structural theory of this notion, the relation between the relative cohomology of groupoids and that of algebroids via van Est maps, and to indicate how it can be used to provide an intrinsic definition of characteristic classes. \end{abstract}
\maketitle
\tableofcontents

\section{Introduction}

The cohomology of Lie groups and Lie algebras has been used successfully to study the interaction between the algebraic and geometric properties of these objects. Lie algebra cohomology was first introduced by \'Elie Cartan in \cite{Cartan:algebras}, and later extended by Chevalley and Eilenberg in \cite{chevalley}. It was originally used to study the topology of Lie groups and homogeneous spaces by connecting de Rham cohomology with the properties of the Lie algebra. On the other hand, Lie group cohomology is closely related to the cohomology of the classifying space (see the work of Stasheff \cite{Stasheff}). Its relation to Lie algebra cohomology has been extensively studied and produced many classical results. In particular, van Est  \cite{vanEstI,vanEstII,vanEstIII} in one of his results produces a map $VE_G$ from the cochain complex that computes the differentiable cohomology of a Lie group $G$ to the Chevalley-Eilenberg complex of its Lie algebra; under some connectivity assumptions on $G$, this induces an isomorphism in cohomology.

Lie groupoids and Lie algebroids are a far-reaching generalization of Lie groups and Lie algebras respectively, and have been used to codify geometric structures and their symmetries (foliations, Poisson manifolds and (other) related structures, PDEs, etc). Their cohomologies are natural extensions of those for Lie groups and Lie algebras, and have been relevant when applied to geometric structures such as foliated cohomology and Poisson cohomology (see, e.g.\ \cite{Victor,Ieke} for foliated cohomology and \cite{Crainic:vanEst,Normal, Reeb, Zeiser} for Poisson cohomology). The relation between the cohomology of Lie groupoids and Lie algebroids started with the work of Weinstein and Xu in \cite{Weinstein_Xu}, where they first extended the van Est map to a map $VE_G:\sC(G)\to\sC(A)$ from the cochain complex that computes the differentiable cohomology of a Lie groupoid $G$ to the Chevalley-Eilenberg complex of its Lie algebroid $A$. Later Crainic \cite{Crainic:vanEst} produced other versions of van Est results; in {\em op.\ cit.\ }he also defined characteristic classes of representations of Lie algebroids, extending the classical characteristic classes of flat vector bundles \cite{Bismut, kamber}, and the modular class of Poisson manifolds \cite{Evens}. (See also \cite{CrainicFernandes:classes,Fernandes:classes} for characteristic classes of Lie algebroids.)

Motivated by understanding characteristic classes of representations of Lie algebroids intrinsically and by our attempt to develop a parallel theory of characteristic classes of representations of Lie groupoids, we are brought back to understanding the fundamentals of the cohomology theories of Lie groupoids and algebroids. One element that was notably missing in the literature was the notion of {\em relative cohomology} for Lie groupoids and Lie algebroids. The main aim of this paper is to introduce and study this notion, as well as to indicate how it can be used to define characteristic classes for both Lie groupoids and Lie algebroids. (These will be studied in more detail in future work.) This approach generalizes that for Lie groups and Lie algebras, where relative cohomology is used to define characteristic classes of flat bundles in various contexts (see \cite[Chapter 6]{kamber}). Throughout we pay special attention to ``van Est type" results, which relate the notions at the global (Lie groupoid) and infinitesimal (Lie algebroid) level. Using the language of groupoids, we find novel and explicit constructions. In what follows, we illustrate the structure of the paper and the main results.\\

{\it Relative cohomology of Lie groupoids (Section \ref{section:rel-group}).} We begin by revisiting the cochain complex $\sC(G)$ of a Lie groupoid $G$ computing its differentiable cohomology. Associated to any wide Lie subgroupoid $K$, we define a cochain subcomplex $\sC(G)^K$ (and its normalized version), which at degree $p$ consists of cochains invariant under $p+1$ natural $K$-actions. For $K$ proper, one can take the average of the actions and produce a cochain map
\begin{equation}\label{intro:Av}\mathrm{Av}\colon \sC(G)\to \sC(G)^K,\end{equation}
 which turns out to be a homotopy equivalence with the inclusion $\iota\colon \sC(G)^K\to  \sC(G)$ its homotopy inverse. Similar results hold when restricting to  the normalized subcomplexes (see Theorem \ref{thm:averaging}). If $G$ is a Lie group and $K$ is a {\em finite} subgroup of $G$, Guichardet constructed a cochain map analogous\footnote{In \cite{Guichardet}, the analog of $\mathrm{Av}$ maps into the normalized version.} to \eqref{intro:Av} that is a homotopy equivalence (see \cite[Proposition 3.3]{Guichardet}). Our construction, while inspired by that of Guichardet, obtains more general results even in the context of Lie groups (see Remark \ref{rmk:integration_groups}).

{\it Relative cohomology of Lie algebroids (Section \ref{section:rel-algebroid}).} In order to define relative cohomology in the context of Lie algebras and Lie groups, one deals with the $\k$-basic subcomplex $\sC(\g)_{\k-\mathrm{basic}}$ of the Chevalley-Eilenberg complex $\sC(\g)$ of a Lie algebra $\g$ and a Lie subalgebra $\k$, consisting of cochains that are $\k$-horizontal and $\k$-invariant under the coadjoint action when restricted to $\k^*$. Similarly, when $\g$ and $\k$ are the Lie algebras of a Lie group $G$ and a Lie subgroup $K$ respectively, one considers the $K$-basic subcomplex $\sC(\g)_{K-\mathrm{basic}}$; in this case, the relevant action is the restriction of the coadjoint action on $\g^*$ to $K$. Generalizing these notions to Lie algebroids and Lie groupoids is not obvious as, in general, a Lie algebroid $A$ does not act on itself, nor does a Lie groupoid act on its Lie algebroid. To solve this issue, we use the key observation that when $B\subset A$ is a Lie subalgebroid over the same base, there is a natural action of $B$ on the quotient $A/B$; similarly, if $K\subset G$ is a Lie subgroupoid integrating $B$, then $K$ acts naturally on $A/B$. This allows us to define and give the main properties of the $B$-basic subcomplex and the $K$-basic subcomplex respectively 
\[\sC(A)_{B-\mathrm{basic}}\subset\sC(A),\ \ \sC(A)_{K-\mathrm{basic}}\subset \sC(A),\]
which coincide with the standard notions in the context of Lie algebras and Lie groups.

{\it The transitive case (Section \ref{transitive}).} Here we assume the existence of a wide transitive Lie subgroupoid $K\subset G$. We show that there is a natural extension map
    \[\mathrm{ext}\colon \sC^\bullet(G_z)^{K_z}\to \sC^\bullet(G)^{K},\]
where $K_z$, $G_z$ are the isotropy Lie groups of $K$ and $G$ respectively, over a point $z$ of the base. This map is an isomorphism of chain complexes with inverse given by the restriction (see Theorem \ref{thm:ext}). An analogous result holds true (see Theorem \ref{thm:ext_2}) when looking at the infinitesimal counterpart:
   \[\mathrm{ext}\colon \sC^\bullet(\g_z)_{K_z\mathrm{-basic}}\to \sC^\bullet(A)_{K\mathrm{-basic}},\]
where $A$ is the Lie algebroid of $G$ and $\g_z$ is the isotropy Lie algebra. Our theory shows that, in this case, the relative cohomology of Lie groupoids and Lie algebroids is the same as the relative cohomology of their isotropy Lie groups and Lie algebras.

{\it Van Est maps (Section \ref{section:van-est}).} In this section we study the van Est maps in the context of relative cohomology, and pay special attention to the case where the Lie subgroupoid $K$ is proper. We make use of {\it tubular structures} and {\it Cartan decompositions} of $(G,K)$ (see Definitions \ref{def:tubular} and \ref{def:Cartan}) to obtain van Est integration maps (i.e., cochain maps going from the relative cohomology of the Lie algebroid to the relative cohomology of the Lie groupoid). As a consequence, we conclude that the van Est differentiation map (see equation \eqref{eq:vanEst_relative})
\[VE_{G/K}\colon \sC^\bullet(G)\to\sC^\bullet(A)_{K\mathrm{-basic}}\]
is a homotopy equivalence when $K$ is a proper subgroupoid  (see Theorem \ref{prop:rightinversegk}). For example, for a Lie group $G$ with finitely many connected components and $K\subset G$ a maximal compact subgroup, tubular structures always exist (see e.g.\ \cite[Chapter VII]{Borel}). To prove the results we use the Perturbation Lemma,  treated in the Appendix. We apply the previous results to conclude that in the transitive case, the differentiable cohomology of the Lie groupoid is homotopy equivalent to the cohomology of its isotropy Lie algebra relative to the maximal compact subgroup of its isotropy Lie group (see Theorem \ref{thm:transitive}).

{\it Characteristic classes (Section \ref{section:cc}).} As an application of our theory we introduce the definition of characteristic classes of a representation of a Lie groupoid. 
This also allows us to give an alternative, simple and direct definition
of the existing notion of characteristic classes of a representation of a Lie algebroid (see \cite{Crainic:vanEst,CrainicFernandes:classes,Fernandes:classes}), and to relate them with our definition of characteristic classes for Lie groupoids. 
This represents the first step to study characteristic classes of Lie groupoids endowed with a geometric structure, e.g., symplectic, symplectic-Nijenhuis, presymplectic, Pfaffian, or contact groupoids. 
At the infinitesimal level, we also hope that our approach will bring a new insight into existing characteristic classes of Poisson manifolds (the modular class), foliations (Bott's characteristic classes), and other types of geometries encoded by Lie algebroids. This will be studied in a future paper.\\

\subsection*{Conventions} All the groupoids appearing in this paper are assumed to be Hausdorff. Moreover, any Lie subgroupoid $K$ of a Lie groupoid $G$ is assumed to be wide, i.e.\  the inclusion $K\hookrightarrow G$ is an injective immersion of Lie groupoids 
 and $K$ has the same base as $G$. Analogously, any Lie subalgebroid $B$ of a Lie algebroid $A$ has the same base as $A$.

\subsection*{Acknowledgments} The author would like to thank Marius Crainic, Ioan M\u arcut and Eckhard Meinrenken for useful comments and suggestions. The author was partly supported by CNPq grant PQ1 304410/2020-9, FAPERJ grants JCNE 262012932021 and ARC
26211412201, and by the Serrapilheira Institute grant Serra-1912-3205. The author is grateful to the Max Planck Institute for Mathematics in Bonn for its hospitality and financial support. This study was financed in part by the Coordena\c c\~ao de Aperfei\c coamento
de Pessoal de N\'ivel Superior -- Brazil (CAPES) -- Finance code 001.

%
\section{Relative cohomology of Lie groupoids} \label{section:rel-group}

We will review the cochain complexes for Lie groupoids, and define the relative cohomology with respect to a Lie subgroupoid. We will give the main properties of these notions and exhibit equivalences between cochain complexes. In Theorem \ref{thm:averaging}, we show that the cohomology of a Lie groupoid relative to a proper subgroupoid computes its differentiable cohomology.

\subsection{The simplicial manifold $B_pG$}\label{sec:simplicial}
Let $G\rra M$ be a Lie groupoid, with source and target denoted by $\sz,\tz\colon G\to M$; unit denoted by $\uz\colon M\to G$ or simply by elements of $M$; inverse denoted by $\iz\colon G\to G$ or simply by $\iz(g)=g^{-1}$. Elements $g,h\in G$ are \emph{composable} if $\sz(g)=\tz(h)$; in this case, the multiplication is denoted by $\mz\colon G\tensor[_\sz]{\times}{_\tz} G\to G$ or simply by $\mz(g,h)=gh$. We denote by 
\[ B_pG=\{(g_1,\ldots,g_p)|\ \sz(g_i)=\tz(g_{i+1}),\ \ 0< i<p\}\]
the space of \emph{$p$-composable arrows}; by convention $B_0G=M$. Every $p$-composable arrow comes with $p+1$ base points $(m_0,\ldots,m_p)$, where $m_i=\sz(g_i)=\tz(g_{i+1})$. 
The collection of spaces $B_pG$ defines a simplicial manifold $B_\bullet G$ called the \emph{nerve} of the groupoid. The face map
$\partial_i\colon B_pG\to 
B_{p-1}G$
drops the $i$-th base point: 
\begin{equation}\label{eq:face}  \partial_i(g_1,\ldots,g_p)=\begin{cases}
(g_2,\ldots,g_p), & i=0,\\
(g_1,\ldots,g_ig_{i+1},\ldots,g_p),&0<i<p,\\
(g_1,\ldots,g_{p-1}), &i=p.
\end{cases}
\end{equation}
The $i$-th degeneracy map $\epsilon_i\colon B_pG\to B_{p+1}G$ doubles the $i$-th base point:
\begin{equation}\label{eq:degeneracy}\epsilon_i(g_1,\ldots,g_p)=(g_1,\ldots,g_{i},m_i,g_{i+1},\ldots,g_p),\ \ i=0,\ldots,p.\end{equation}
The manifolds $B_pG$ also
come equipped with 
$p+1$ commuting right $G$-actions $A_i$, $i=0, \ldots , p$, defined by
\begin{equation}\label{eq:actions} 
 A_i((g_1,\ldots,g_p), a)=\begin{cases}
(a^{-1}g_1,g_2,\ldots,g_p) & i=0,\\
(g_1,\ldots,g_i a,a^{-1} g_{i+1},\ldots,g_p)&0<i<p,\\
(g_1,\ldots,g_{p-1},g_p a) &i=p;
\end{cases}
\end{equation}
the $i$-th action happens along the map $(g_1,\ldots,g_p)\mapsto m_i$. 
(For basics on principal actions of Lie groupoids, see e.g.\ \cite{Crainic:vanEst}.)\\

We define a simplicial principal bundle
\[\kappa_\bullet\colon E_\bullet G\to B_\bullet G\]
over $BG$ as follows. 
For each $p$, define the trivial principal $G$-bundle
\begin{equation}\label{eq:kappap}
 \kappa_p\colon E_p G=B_pG\times_M G\to B_pG,\end{equation}
with principal left action 
\begin{equation}\label{eq:Gaction} a\cdot (g_1,\ldots,g_p;g)=(g_1,\ldots,g_p;ga^{-1}),\end{equation}
that happens along the map
\[ \pi_p(g_1,\ldots,g_p;g)=\sz(g).\] 
After the change of coordinates $(g_1,\ldots,g_p;g)$ $\mapsto$ $ ((g_1,g_2\cdots g_pg),\ldots,(g_{p-1},g_pg),(g_p,g))$, one can also represent $E_pG$ as
$$B_pG\times_M G\simeq B_p(G\ltimes G),$$
where $G\ltimes G\rra G$ is the action groupoid, with the $G$-action on $G$ by left multiplication. This is useful as one readily sees that $E_\bullet G$ carries a simplicial structure.\\
Yet another description of $E_pG$ shows that $\kappa_\bullet$ is a simplicial map: After the change of coordinates $(g_1,\ldots,g_p;g)\mapsto (g_1g_2\cdots g_pg,g_2\cdots g_pg,\ldots,g_pg,g)$, one can represent $E_pG$ as 
\begin{equation}\label{eq:Ep}E_pG\simeq\{(g_0,\ldots,g_p)\in G^{p+1}\mid \sz(g_0)=\cdots=\sz(g_p)\}.\end{equation}
Under this identification, the projection map $\kappa_p$ of \eqref{eq:kappap} is given by 
\[(g_0,\ldots,g_p)\mapsto(g_0g_1^{-1},g_1g_2^{-1},\ldots, g_{p-1}g_p^{-1} ),\]
the principal $G$-action \eqref{eq:Gaction} is given by 
$a\cdot(g_0,\ldots,g_p)=(g_0a^{-1},\ldots,g_pa^{-1})$, the face maps $\partial_i\colon E_pG\to E_{p-1}G$ are given by dropping the $i$-th entry, and the degeneracy maps $\epsilon_i\colon E_pG\to E_{p+1}G$ are given by repeating the $i$-th entry. Hence, the face and degeneracy maps are $G$-equivariant, and so $\kappa_\bullet$ is a simplicial map.\\
The actions \eqref{eq:actions} lift to
$p+1$ commuting right $G$-actions $\bar{A}_i$, $i=0, \ldots , p+1$, on $E_pG=B_pG\times_M G$, given by
\begin{equation}\label{eq:actions2} 
\bar{A}_i((g_1,\ldots,g_p;g), a)=\begin{cases}
(a^{-1}g_1,g_2,\ldots,g_p;g) & i=0,\\
(g_1,\ldots,g_ia, a^{-1}g_{i+1},\ldots,g_p;g)&0<i<p,\\
(g_1,\ldots,g_{p-1},g_p a;a^{-1}g) &i=p;
\end{cases}
\end{equation}
the $i$-th action happens along the map $(g_1,\ldots,g_p;g)\mapsto m_i$. 

\subsection{The groupoid complex}
The groupoid cochain complex $\big(\sC(G),\delta\big)$ is given by
\[\xymatrix{0 \ar[r] & \sC^0(G) \ar[r]^-{\delta} & \cdots \ar[r]^-{\delta} & \sC^p(G) \ar[r]^-{\delta} & \sC^{p+1}(G) \ar[r]^-{\delta} & \cdots,}\]
where $\sC^p(G)=C^\infty(B_pG)$ and $\delta$ is given on $p$-cochains by $\delta = \sum_{i=0}^{p+1}(-1)^i \partial_i^*$. A direct calculation shows that $\delta\circ\delta=0$. 
Its cohomology is called the {\it differentiable cohomology of} $G$.

For the simplicial manifold $E_\bullet G$ as in \eqref{eq:kappap}, we denote its cochain complex by $\big(\sC(G\ltimes G),\delta\big)$. Note that the inclusion $\sj\colon\sC(G)\to \sC(G\ltimes G)$ given in degree $p$ by the pullback $\kappa_p^*$ identifies $\sC(G)$ with the $G$-invariant subcomplex of $\sC(G\ltimes G)$ with respect to the action induced by \eqref{eq:Gaction}. \\

 Consider now a Lie subgroupoid $K\subset G$.
For each of the actions \eqref{eq:actions} we can consider the restriction to $K$-actions. For each integer $p \geq 0$, we denote the subspace of $\sC^p(G)$ consisting of $p$-cochains that are invariant w.r.t.\ $p+1$ commuting $K$-actions by $\sC^p(G)^K$. By a direct calculation, we have that $\partial_i^*(\sC^p(G)^K) \subseteq \sC^{p+1}(G)^K$ for any $p \geq 0$ and any $i=0,\ldots,p+1$.
\begin{definition}The {\it Lie groupoid complex relative to $K$} is the $K$-invariant subcomplex of 
\[\big(\sC(G)^{K},\delta \big) \subseteq \big(\sC(G),\delta \big).\]
The {\em (differentiable) cohomology of $G$ relative to $K$} is the cohomology of this subcomplex; we denote it by $H^\bullet(G,K)$. If $K$ is clear from the context, we refer to this cohomology simply as the {\em relative cohomology of $G$}.
\end{definition}

\begin{remark}\label{rmk:normalized} Sometimes it is useful to work with the {\it normalized subcomplex}
$$\widetilde{\sC}(G)\subset \sC(G)$$ consisting of cochains with the property that $f(g_1\ldots,g_p)=0$ whenever $g_i$ is a unit for some $i$. It is well-known that the inclusion $\iota\colon  \widetilde{\sC}(G)\to\sC(G)$ is a homotopy equivalence (see e.g.\ \cite{Guichardet}).

For a Lie subgroupoid $K\subset G$, we analogously define the {\it normalized $K$-invariant subcomplex} $$\widetilde{\sC}(G)^{K}$$ consisting of functions $f\in \sC(G)^{K}$ with the property that $f(g_1,\ldots,g_p)=0$ whenever $g_i \in K$ for some $i$. Observe that $ \widetilde{\sC}(G)^{K}= \widetilde{\sC}(G)\cap\sC(G)^K$.
\end{remark}

Similarly, one defines the normalized subcomplex $\widetilde{\sC}(G\ltimes G)\subset \sC(G\ltimes G)$, the $K$-invariant subcomplex $\sC(G\ltimes G)^K$, and its normalized version $\widetilde{\sC}(G\ltimes G)^K$
\[\widetilde{\sC}(G\ltimes G)^K\subset \sC(G\ltimes G)^K\subset \sC(G\ltimes G)\] 
with respect to the $K$-actions \eqref{eq:actions2}. In particular, a function $f\in \sC(G\ltimes G)^K$ is normalized if $f(g_1,\ldots,g_p;g)=0$ whenever $g_i\in K$ for some $i=1,\ldots p$. Once again, the map $\sj\colon\sC(G)^K\to \sC(G\ltimes G)^K$ ($\sj=\kappa_p^*$ at degree $p$) includes $\sC(G)^K$ as the $G$-invariant subcomplex of $\sC(G\ltimes G)^K$ with respect to the action \eqref{eq:Gaction}, and similarly for $\sj\colon\widetilde{\sC}(G)^K\to \widetilde{\sC}(G\ltimes G)^K$.

\begin{remark}\label{rmk:normalized_2}For latter use, in terms of the description \eqref{eq:Ep} of $EG$,
 the normalized subcomplex $\widetilde{\sC}(G\ltimes G)\subset \sC(G\ltimes G)$ consists of those functions on $E_pG$ with the property that $f(g_0\ldots,g_p)=0$ whenever $g_i=g_{i+1}$ for some $i=0,\ldots,p-1 $. The $K$-invariant subcomplex $\sC(G\ltimes G)^K$ consists of those functions with the property that $f(k_0g_0,\ldots,k_pg_p)=f(g_0,\ldots,g_p)$ whenever $k_i\in K$ and the multiplication $k_ig_i$ is well defined for $i=0,\ldots,p-1 $. Lastly, its normalized version $\widetilde{\sC}(G\ltimes G)^K$ consists of function $f\in\sC^p(G\ltimes G)^K$ such that $f(g_0,\ldots,g_p)=0$ whenever $g_ig_{i+1}^{-1}\in K$ for some $i=0,\ldots, p-1$. 
\end{remark}

\begin{remark}\label{rmk:cup}
The usual cup product
$$(f_1\cup f_2)(g_1,\ldots,g_{p+q})=f_1(g_1,\ldots,g_p)f_2(g_{p+1},\ldots, g_{p+q})$$ 
defines a product structure $\sC^p(G)\times \sC^q(G)\to \sC^{p+q}(G)$, which passes to cohomology. By a direct calculation, the cup product of two elements in $\sC^\bullet(G)^K$ lies in $\sC^\bullet(G)^K$. Hence, $\big(\sC^\bullet(G)^K,\delta\big)\subset \big(\sC^\bullet(G),\delta\big)$ is a DG subalgebra and $H^\bullet(G,K)$ (and $H^\bullet(G)$) is a graded algebra. 
Similarly, the normalized $K$-invariant subcomplex is a DG subalgebra and its cohomology is a graded algebra.
\end{remark}

\subsection{Equivalences between complexes}\label{sec:equivalences}
 Let $K$ be a Lie subgroupoid of $G$. The aim of this section is to show that, if $K$ is proper, then the cohomology of $G$ can be computed using the normalized cochain complex of $G$ relative to $K$ (see Theorem \ref{thm:averaging}). We do this in two steps. First, we prove Theorem \ref{thm:normalized} below, which extends to the relative case the fact that the normalized cochain complex computes the differentiable cohomology. The second step is to prove Theorem \ref{thm:averaging}.



\begin{theorem}\label{thm:normalized} The inclusion $\iota\colon \widetilde{\sC}(G)^{K}\to \sC(G)^{K} $ is a homotopy equivalence.
\end{theorem}

In preparation for the proof of Theorem \ref{thm:normalized}, recall that the cohomology of
$\big(\sC(G\ltimes G),\delta\big)$ vanishes in all positive degrees.
To see this, consider the maps
\begin{equation*}\label{eq:maph} h_p\colon E_pG\to E_{p+1}G,\ (g_1,\ldots,g_p;g)\mapsto (g_1,\ldots,g_p,g;m)\end{equation*}
where $m=\sz(g)$. 
The map 
\begin{equation}\label{eq:h} \mathsf{h}\colon \sC^{p}(G\ltimes G)\to \sC^{p-1}(G\ltimes G),\ \ f\mapsto (-1)^p h_{p-1}^*f\end{equation}
satisfies 
$[\mathsf{h},\delta]=1-\si\circ \sp,$
where $\si=\pi_0^*\colon C^\infty(M)\to\sC^{0}(G\ltimes G) $ and $\sp\colon \sC^{0}(G\ltimes G)\to C^\infty(M)$ is the 
left inverse to $\si$
given by pullback under the inclusion $\uz\colon M\hookrightarrow E_0 G=G.$ Similarly, the cohomology of $\big(\sC(G\ltimes G)^K,\delta\big)$ vanishes in all positive degrees: To see this, use the restrictions of $\mathsf{h}, \mathsf{p}, \mathsf{i}$ to this complex.

\begin{proof}[Proof of Theorem \ref{thm:normalized}]
First, we will proceed to exhibit a homotopy inverse $N\colon \sC(G)\to  \widetilde{\sC}(G)$ of the inclusion $\iota\colon  \widetilde{\sC}(G)\to\sC(G)$. We will show that $N$ restricts to a homotopy inverse $N^K\colon {\sC}(G)^{K}\to  \widetilde{\sC}(G)^{K}$ of $\iota\colon \widetilde{\sC}(G)^{K}\to \sC(G)^{K} $.

For the definition of $N$ (and that of the homotopy $\mathsf{n}$ below), we use the description of $EG$ presented in \eqref{eq:Ep} and follow Remark \ref{rmk:normalized_2}.
Setting 
\[(\delta_{g_0}\otimes\cdots\otimes \delta_{g_p})(f):=f(g_0,\ldots, g_p),\]
as in \cite[Chapter 2]{Guichardet}, we define $N\colon \sC^p(G\ltimes G)\to  \widetilde{\sC}^p(G\ltimes G)$ by
\[N(f)(g_0,\ldots,g_p)=\delta_{g_0}\otimes(\delta_{g_1}-\delta_{g_0})\otimes\cdots\otimes(\delta_{g_n}-\delta_{g_{p-1}})(f).\]
(It follows by a direct computation that $N(f)\in\widetilde{\sC}^p(G\ltimes G)$.) Note that $N$ is $G$-equivariant and it restricts to a map $N^K\colon {\sC}(G\ltimes G)^{K}\to  \widetilde{\sC}(G\ltimes G)^{K}=\widetilde{\sC}(G\ltimes G)\cap{\sC}(G\ltimes G)^{K}$. In degree $p=0$, $N=\id_{\sC^0(G\ltimes G)}$ so that $\si=\si\circ N$; a computation shows that $N$ is a cochain map (see also \cite[Corollary 2.4]{Guichardet}). It is also a computation to see that \[N\circ\iota=\id_{\widetilde{\sC}(G\ltimes G)}\] for the inclusion $\iota\colon \widetilde{\sC}^p(G\ltimes G)\to {\sC}^p(G\ltimes G)$. For $\mathsf{v}=\id_{{\sC}(G\ltimes G)}-\iota\circ N$, we will show (see the computation below) that $\mathsf{v}=[\mathsf{n},\delta]$ for some $G$-equivariant maps $\mathsf{n}\colon {\sC}^p(G\ltimes G)\to {\sC}^{p-1}(G\ltimes G)$ that restrict to the $K$-invariant subcomplex ${\sC}(G\ltimes G)^K$, thus concluding that $\iota\colon \widetilde{\sC}(G\ltimes G)^{K}\to \sC(G\ltimes G)^{K}$ is a homotopy equivalence. As all the maps involved are $G$-equivariant, it follows that  $\iota$ restricts to the desired homotopy equivalence $\iota\colon \widetilde{\sC}(G)^{K}\to \sC(G)^{K}$ of the $G$-invariant subcomplexes. 

First, we define $\mathsf{n}_1\colon \sC^1(G\ltimes G)\to \sC^0(G\ltimes G)$ as $\mathsf{n}_1=0$. As $\mathsf{v}_0=0$ then $\mathsf{v}_0=\mathsf{n}_1\circ\delta$. Let us argue recursively and suppose that we have constructed $G$-equivariant maps $\mathsf{n}_i\colon \sC^i(G\ltimes G)\to \sC^{i-1}(G\ltimes G)$ for $i=1,\ldots, p-1$ with the property that 
\[\mathsf{v}_{i-1}=\delta\circ\mathsf{n}_{i-1}+\mathsf{n}_i\circ\delta.\]
Let $\bar{\mathsf{v}}_{p-1}=\mathsf{v}_{p-1}-\delta\circ\mathsf{n}_{p-1}$, then
\[\bar{\mathsf{v}}_{p-1}\circ \delta=\mathsf{v}_{p-1}\circ\delta-\delta\circ\mathsf{n}_{p-1}\circ\delta=\delta\circ\mathsf{v}_{p-2}-\delta\circ\mathsf{n}_{p-1}\circ\delta=0,\]
hence $\bar{\mathsf{v}}_{p-1}|_{\mathrm{Im}(\delta\colon \sC^{p-2}{(G\ltimes G)}\to\sC^{p-1}(G\ltimes G))}=0$. As \[\mathrm{Im}(\delta\colon \sC^{p-2}{(G\ltimes G)}\to\sC^{p-1}(G\ltimes G))=\ker (\delta\colon \sC^{p-1}{(G\ltimes G)}\to\sC^{p}(G\ltimes G)),\] $\bar{\mathsf{v}}_{p-1}$ passes to the quotient $\sC^{p-1}(G\ltimes G)/\ker\delta$. Note that 
\[\delta\colon\sC^{p-1}(G\ltimes G)/\ker\delta\to\sC^p(G\ltimes G)\]
is injective with left inverse $\mathsf{h}\colon \sC^{p}(G\ltimes G)\to \sC^{p-1}(G\ltimes G)/\ker\delta$ (the homotopy operator \eqref{eq:h} after passing to the quotient). Define $\mathsf{n}_p \colon \sC^p(G\ltimes G)\to \sC^{p-1}(G\ltimes G)$ by 
\begin{equation}\label{eq:hom}\mathsf{n}_p(f)(g_0,\ldots,g_{p-1})=\bar{\mathsf{v}}_{p-1}(\mathsf{h}(g_0\cdot f))(\tz(g_0),g_1g_0^{-1},\ldots, g_{p-1}g_0^{-1}).\end{equation}
The map $\mathsf{n}_p$ is $G$-equivariant and $\mathsf{n}_p\circ\delta=\bar{\mathsf{v}}_{p-1}$, thus $\mathsf{n}_p\circ\delta=\mathsf{v}_{p-1}-\delta\circ\mathsf{n}_{p-1}$ as desired.\end{proof}

Assume that $K\subset G$ is a {\it proper} Lie subgroupoid, in the sense that $(\tz,\sz):K\to M\times M$ is a proper map.
We will use the properness of the Lie subgroupoid $K$ of $G$ to produce averaging operators, turning any cocycle on $\sC(G)$ into a $K$-invariant cocycle on $\sC(G)^K$.  
For this purpose recall \cite{Crainic:vanEst,Joao, Tu} that a proper Lie groupoid admits {\it (left invariant) properly supported normalized Haar system}
$\mu=\{\mu_m\}_{m\in M}$, i.e.\ a family of smooth measures on $\tz^{-1}(m)$ such that:

\begin{itemize}
\item for any compactly supported function $f\in C^{\infty}(K)$, the formula
$$m \mapsto\int_{a\in\tz^{-1}(m) }f(a)\ \mu_m(a)$$
defines a smooth function on $M$,
\item $\mu$ is left invariant, i.e.\ $(L_g)_*\mu_m=\mu_{\tz(g)}$ for all $g\in\sz^{-1}(m)$,
\item $\tz$ restricts to a proper map $\mathrm{supp}(\mu)\to M$, and
\item $\int_{\tz^{-1}(m) }\mu_m=1$ for all $m\in M$.
\end{itemize}
Choose such a Haar system $\mu$ on $K$ and for each $p$ denote by 
\begin{equation}\label{eq:av} \on{Av}^{(i)}\colon C^\infty(B^pG)\to C^\infty(B^pG)\end{equation}
the induced averaging operator with respect to the $i$-th $K$-action \eqref{eq:actions}. The operators $\on{Av}^{(i)}$ commute since the actions commute, and we denote by 
\begin{equation}\label{eq:averaging}\on{Av}=\on{Av}^{(0)}\circ \cdots \circ \on{Av}^{(p)}\end{equation}
 the total $K^{p+1}$-averaging operator. 
 
 \begin{theorem}\label{thm:averaging} Suppose $K\subset G$ is a proper Lie subgroupoid. An averaging operator 
 \[\on{Av}\colon \sC^\bullet(G)\to  \sC^\bullet(G)^{K}\]
 defined by \eqref{eq:averaging} is a homotopy equivalence with homotopy inverse given by the inclusion
  \[\iota\colon \sC^\bullet(G)^{K}\to \sC^\bullet(G).\]
  
Hence, the inclusion $\iota\colon \widetilde{\sC}^\bullet(G)^{K}\to \sC^\bullet(G)$ is a homotopy equivalence.
  \end{theorem}

  \begin{proof} As in \eqref{eq:averaging}, denote by 
  \[\on{Av}_{EG}\colon \sC(G\ltimes G)\to \sC(G\ltimes G)\]
   the total averaging operator with respect to the $K$-actions \eqref{eq:actions2}. This averaging operator is a $G$-equivariant (w.r.t. \eqref{eq:Gaction}) cochain map, restricted to the $G$-invariant subcomplex $\sC(G)$. Note that $\on{Av}_{EG}$ is a left inverse of the inclusion $\iota\colon\sC(G\ltimes G)^K\to \sC(G\ltimes G)$. We will see that $\iota\circ\on{Av}_{EG}$ is homotopic to the identity with a $G$-equivariant homotopy $\mathsf{n}\colon\sC^\bullet(G\ltimes G)\to \sC^{\bullet-1}(G\ltimes G) $, i.e.\ $\mathsf{v}:=\id_{\sC(G\ltimes G)}-\iota\circ\on{Av}_{EG}=[\mathsf{n}, \delta]$, thus concluding that on the $G$-invariant complexes $\on{Av}\colon \sC(G)\to  \sC(G)^{K}$ is a homotopy equivalence with $\iota\colon\sC(G)^{K}\to \sC(G)$ a homotopy inverse.
   
The construction of $\mathsf{n}$ is completely analogous to that of the proof of Theorem \ref{thm:normalized}, with $\on{Av}_{EG}$ playing the role of $N$. Adopting the description of $EG$ as in \eqref{eq:Ep}, one defines the base case $\mathsf{n}_1\colon\sC^1(G\ltimes G)\to \sC^0(G\ltimes G)$ by 
\[\mathsf{n}_1(f)(g)=\int_{a\in \tz^{-1}(t(g))\cap K}f(a^{-1}g,g)\ \mu_{\tz(g)}(a).\]
For the inductive case one defines $\mathsf{n}_p$ as in \eqref{eq:hom}. 

To conclude that  $\iota\colon \widetilde{\sC}^\bullet(G)^{K}\to \sC^\bullet(G)$ is a homotopy equivalence we use Theorem \ref{thm:normalized}.
  \end{proof}

\section{Relative cohomology of Lie algebroids}\label{section:rel-algebroid}
In this section, we introduce the Chevalley-Eilenberg complex of a Lie algebroid relative to a subalgebroid. In the case where the algebroid comes from a Lie groupoid $G$, we also define the Chevalley-Eilenberg complex of a Lie algebroid relative to a subgroupoid of $G$. 
These notions respectively extend the standard definitions of the $
\k$-basic and $K$-invariant
subcomplexes, when $\g$ is a Lie algebra, $
\k\subset\g$ is a subalgebra, and $K$ is a Lie subgroup integrating $\k$.

\subsection{The Chevalley-Eilenberg complex and its geometry}
Let $A\Ra M$ be a Lie algebroid with anchor map $\a\colon A\to TM.$ Let $\big(\sC(A),\d_{CE}\big)$ be the \emph{Lie algebroid complex} or \emph{Chevalley-Eilenberg complex}, with $p$-cochains 
\[ \sC^p(A)=\Gamma(\wedge ^p A^* ),\] 
and with Chevalley-Eilenberg differential $\d_{CE}\colon \sC^p(A)\to \sC^{p+1}(A)$ defined by
\begin{multline*}
\big(\d_{CE}\alpha\big)(\xi_1,\ldots,\xi_{p+1})=\sum_{i}(-1)^{i+1}\L_{\a(\xi_i)}\big(\alpha(\xi_1,\ldots, \hat \xi_i,\ldots,\xi_{p+1})\big)+\\
\sum_{i<j}(-1)^{i+j}\alpha([\xi_i,\xi_j],\xi_1,\ldots,\hat \xi_i,\ldots,\hat \xi_j,\ldots,\xi_{p+1}),
\end{multline*}
where $\alpha\in\Gamma(\wedge ^p A^* )$ and $\xi_1,\ldots,\xi_{p+1}\in\Gamma(A)$. For $\xi\in\Gamma(A)$, we denote by 
\[\iota_\xi\colon\sC^\bullet(A)\to \sC^{\bullet-1}(A),\ \ \L_\xi\colon\sC^\bullet(A)\to\sC^\bullet(A)\]
the operators given by contraction and the Lie derivative $\L_\xi=\d_{CE}\iota_\xi+\iota_\xi\d_{CE}$, respectively. On $\sC^0(A)=C^{\infty}(M)$, we have $\L_\xi=\L_{\a(\xi)}.$\\

Assume now that $A$ is the Lie algebroid of $G\rra M$. Thus, $A$ is the vector bundle whose sections are the left-invariant vector fields of $G$ (tangent to the $\tz$-fibers); for $\xi\in\Gamma(A)$ we denote by $\xi^L$ the corresponding left-invariant vector field. The anchor map $\a\colon A\to TM$ is characterized by the property that $\xi^L\sim_\sz -\a(\xi)$.  Denote by $\big(\Omega_\mathcal{F}(G),\d_{Rh}\big)$ the de Rham complex of $\tz$-foliated forms of $G$, and denote by $\left(\Omega_\mathcal{F}(G)^L,\d_{Rh}\right)$ the subcomplex of left-invariant forms. For $\alpha\in \sC(A)$ we denote by $\alpha^L$ the left-invariant foliated form on $G$. This operation intertwines the Chevalley-Eilenberg and the de Rham differentials, and the contraction $\iota_\xi$ and Lie derivative $\L_\xi$ operators of $\sC(A)$ with the usual contraction of vector fields $\iota_{\xi^L}$ and Lie derivative of vector fields $\L_{\xi^L}$, respectively. Thus
\begin{equation}\label{eq:forms}
\begin{array}{ccc}
\big(\sC(A),\d_{CE}\big) & \simeq & \big(\Omega_\mathcal{F}(G)^L,\d_{Rh}\big)\\[3pt]
\xi & \longleftrightarrow & \xi^L
\end{array}
\end{equation}
\begin{equation*}\xymatrix{  
  \sC(A)  \ar[r]^-{\iota_\xi } \ar@{-}[d]_{\simeq}&   \sC(A) \ar@{-}[d]^{\simeq}  \\
\Omega_\mathcal{F}(G)^L \ar[r]^-{ \iota_{\xi^L}}& \Omega_\mathcal{F}(G)^L 
 }\qquad\qquad
 \xymatrix{  
  \sC(A)  \ar[r]^-{\L_\xi } \ar@{-}[d]_{\simeq}&   \sC(A) \ar@{-}[d]^{\simeq}  \\
\Omega_\mathcal{F}(G)^L \ar[r]^-{ \L_{\xi^L}}& \Omega_\mathcal{F}(G)^L 
 }
 \end{equation*}
Actually, the Lie derivatives arise as the infinitesimal generators of an action of $\Bis(G)$ -- the group of bisections of $G$:
$$\Bis(G)=\{b\colon M\to G\mid \tz\circ b= \mathrm{id},\ \sz\circ b\mathrm{ \ a \ diffeomorphism}\};$$
if we only require that the map $b$ is defined in an open subset of $M$, then we are talking about a local bisection and we denote by $\Bis_\mathrm{loc}(G)$ the set of local bisections. The product and inverse of $\Bis(G)$ is given by
\begin{equation}\label{eq:prod_bi}
    (b\cdot a)(m)=b(m)a(\sz\circ b(m)),\ \ b^{-1}(m)=\iz(b((\sz\circ b)^{-1}(m))), \ \ \text{for }m\in M.\end{equation}
For local bisections, one has a partially defined product: If $a,b$ are local bisections and the domain of $a$ is equal to the image of $\sz\circ b$, the formula $b\cdot a$ defines a local bisection.\\
The group of bisections acts on the left on $\sC(A)$ by  
\begin{equation}\label{eq:action_Bis}b\cdot\alpha :=C_b^*(\alpha), \ \ \  b\in\Bis(G),  \ \alpha\in\sC(A),\end{equation}
where $C_b$ is the diffeomorphism of the Lie groupoid given by conjugation with $b$
\begin{equation}\label{eq:conjugation}C_b\colon G\to G, \ g\mapsto b(\tz(g))^{-1}gb(\sz(g)). \end{equation}
If $b$ is a local bisection, \eqref{eq:action_Bis} still makes sense but only as a section defined on the domain of $b$.
Note that $C_b^*$ in \eqref{eq:action_Bis} is just the pullback of the dual of the Lie algebroid isomorphism $\mathrm{Lie}(C_b)=dC_b|_A$ induced by $C_b$, and therefore,
\begin{equation}\label{eq:Bis_d} b\cdot(\d_{CE}\alpha)=\d_{CE}(b\cdot\alpha).\end{equation}
By differentiating the action \eqref{eq:action_Bis} one obtains an operator $\nabla\colon\Gamma(A)\times\sC^\bullet(A)\to \sC^\bullet(A)$ 
\begin{equation}\label{eq:operator}\nabla_\xi(\alpha)=\frac{d}{d\epsilon}\mid_{\epsilon=0}\varphi^\epsilon_\xi\cdot\alpha\ ;\end{equation}
here $\varphi_\xi^\epsilon$ is the family of bisections defined by the flow of $\xi^L$ as $\varphi_\xi^\epsilon(m)=\varphi_{\xi^L}^\epsilon(m),\ m\in M$. (Strictly speaking, the r.h.s. of \eqref{eq:operator} only makes sense locally. This is not a problem because everything still works at the level of local bisections.) Note that by the left-invariance of $\xi^L$,
\begin{equation}\label{eq:left-invariant}\varphi_{\xi^L}^\epsilon(g)=g \varphi_{\xi}^\epsilon(\sz(g)).\end{equation}

\begin{lemma}\label{cor:final} Let $G$ be a Lie groupoid with Lie algebroid $A$. For $\alpha\in\sC^\bullet(A)$ and $\xi\in \Gamma(A)$,
$$\L_\xi(\alpha)=\nabla_\xi(\alpha). $$
\end{lemma}

\begin{proof} 
One has that 
 \begin{equation}\label{eq:ref}( \varphi_{\xi}^\epsilon\cdot \alpha)|_m=( \varphi_{\xi^L}^\epsilon)^*(\alpha^L)|_m, \ \ \textrm{for }m\in M.\end{equation}
Indeed, by \eqref{eq:left-invariant}, the r.h.s.\ of the above equation evaluated on a vector $X\in T_mG\cap\ker d\tz$ equals
\begin{equation*}
\alpha^L(d\mz(X,d\varphi^\epsilon_\xi(d\sz(X))))=\alpha^L(d\mz(0_{\varphi^\epsilon_\xi(m)^{-1}},d\mz(X,d\varphi^\epsilon_\xi(d\sz(X))))),
\end{equation*}
where in the equality we use the left-invariance of $\alpha^L$,
while the l.h.s.\ evaluated on $X$ equals
\begin{equation*}
\alpha(d\mz(d\iz(d\varphi^\epsilon_\xi(d\tz(X))),d\mz(X,d\varphi^\epsilon_\xi(d\sz(X)))))=\alpha(d\mz(0_{\varphi^\epsilon_\xi(m)^{-1}},d\mz(X,d\varphi^\epsilon_\xi(d\sz(X))))).
\end{equation*}
Since $\alpha_m=\alpha^L_m$ for any $m\in M$, \eqref{eq:ref} follows.
Hence, by the usual formula that uses flows of vector fields,
$$\nabla_\xi(\alpha)(m)=(\L_{\xi^L}\alpha^L)(m), \ \ \textrm{for }m\in M.$$ 
The result follows by \eqref{eq:forms}.
 \end{proof}

An alternative way to understand the action of $\Bis(G)$ on $\sC(A)$ and the operator $\nabla$, is to consider the Lie groupoid of first jets of local bisections of $G$
\begin{equation*}
    J^1G\rra M.
\end{equation*}
The structural maps of $J^1G$ are given by 
\[\sz(j_m^1b)=\sz(b(m)),\quad \tz(j_m^1b)=\tz(b(m))=m, \quad j_m^1b\cdot j_{m'}^1a=j^1_{m}(b\cdot a),\] 
where $m':=\sz(b(m))$ and $b\cdot a$ the product of local bisections \eqref{eq:prod_bi}.
The natural projection $J^1G\to G,\ j^1_mb\mapsto b(m)$ is a morphism of Lie groupoids. The Lie algebroid $A$ of $G$ is a representation of $J^1G$: An element $j^1_mb\in J^1G$ acts on the right on $\xi\in A_m$ by
\begin{equation}\label{eq:jet_action}\xi\cdot j^1_mb= d\mz(0_{b(m)^{-1}},d\mz(\xi,d_mb(-\a(\xi))))\in A_{\sz(b(m))}.\end{equation}
Note that \eqref{eq:jet_action} is equal to $d_mC_b(\xi)$ (see \eqref{eq:conjugation}).

Consider the induced representations of $J^1G$ on $A^*$, given by the dual (left) action on $A^*$, and its exterior powers. The action \eqref{eq:action_Bis} of a (local) bisection $b$ on an element $\alpha\in \sC(A)$ can now be understood in terms of these representations as
\begin{equation}\label{eq:equal}
    (b\cdot \alpha)_m=j^1_mb\cdot \alpha_{\sz(b(m))},\ m\in M.
\end{equation}
Fix an integer $p \geq 0$. By differentiating the representation of $J^1G$ on $\Lambda^p A^*$, one obtains the infinitesimal representation $\tilde\nabla:\Gamma(J^1A)\times \sC^p(A)\to \sC^p(A)$ of $J^1A$ -- the Lie algebroid of $J^1G$. In analogy with \eqref{eq:equal}, the defining formula \eqref{eq:operator} for the action of $\Gamma(A)$ on $\sC(A)$ can now be interpreted in terms of these infinitesimal representations as 
\[\nabla_\xi(\alpha)=\tilde \nabla_{j^1\xi}(\alpha), \quad \mathrm{for }\ \xi\in \Gamma(A), \ \alpha\in \sC^\bullet(A).\]


\subsection{The Lie algebroid complex relative to a subgroupoid}\label{subsection:3.2}

Assume that  $A$ is the Lie algebroid of $G$ and that $B$ is the Lie algebroid of a Lie subgroupoid $K\subset G$. Thus, $B\subset A$ is a Lie subalgebroid over the same base $M$. The $B$-horizontal subspace is defined as 
\begin{equation}\label{eq:hor}
    \sC(A)_{B\mathrm{-hor}}:=\cap_{\eta\in\Gamma(B)}\ker\iota_\eta,\end{equation}
which, in degree $p$, we can think of naturally as the space of sections of $\wedge^p(A/B)^*$. 

To talk about the $K$-invariant elements of $\sC(A)_{B\mathrm{-hor}}$, we need the following digression: When restricting to $J^1K$ the representation \eqref{eq:jet_action} on $A$, it follows that $B$ is a subrepresentation, and therefore, the quotient $A/B$ becomes a representation of $J^1K$. Moreover, if $j^1_ma,j^1_mb\in J^1K$ are so that $a(m)=b(m)$, then
\begin{equation*}\label{eq:descend}
    \xi\cdot j^1_ma=\xi\cdot j^1_mb, \quad\xi\in (A/B)_m.
\end{equation*}
This implies that the action of $J^1K$ comes from a representation of $K$ on $A/B$. By taking the representation on the dual $(A/B)^*$ and its exterior powers, we obtain $K$-representations on $\wedge^\bullet(A/B)^*.$
In terms of bisections, given $k\in K$ and $\alpha\in \sC^\bullet(A)_{B\mathrm{-hor}}$,
\begin{equation}\label{eq:action_K}k\cdot\left(\alpha_{\sz(k)}\right)=\big(b\cdot \alpha\big)_{\tz(k)},\end{equation}
 where $b\in\Bis_{\mathrm{loc}}(K)$ is any local bisection with the property that $b(\tz(k))=k$, and the r.h.s.\ is as in \eqref{eq:action_Bis}.

\begin{definition}The {\it Lie algebroid complex relative to $K$} is the $K$-basic subcomplex defined as
\begin{equation}\label{eq:basic}\sC(A)_{K\mathrm{-basic}}:=\{\alpha\in\sC(A)_{B\mathrm{-hor}}\mid k\cdot\alpha_{\sz(k)}=\alpha_{\tz(k)}\mathrm{ \ for\  all\ } k\in K\}.\end{equation}
The {\em cohomology of $A$ relative to $K$} is the cohomology of $\big(\sC(A)_{K\mathrm{-basic}},\d_{CE}\big)$; we denote it by $H^\bullet(A,K)$. If $K$ is clear from the context, we refer to this cohomology simply as the {\em relative cohomology of $A$}.
\end{definition}
We postpone the proof that \eqref{eq:basic} defines a subcomplex to Proposition \ref{prop:subcomplex}.


\subsection{The Lie algebroid complex relative to a subalgebroid}
Consider now a Lie subalgebroid $B\subset A$ over the same base $M$. The $B$-horizontal subspace \[\sC(A)_{B\mathrm{-hor}}\subset \sC(A)\] is defined as in \eqref{eq:hor} (which we think of as the space of sections of $\wedge^\bullet(A/B)^*$).
By the identity $[\L_\eta,\iota_\xi]=\iota_{[\eta,\xi]}$, one checks that the Lie derivative induces $B$-representations on $\wedge^\bullet(A/B)^*$ with flat connections
\begin{equation}\label{eq:ad}\L\colon\Gamma(B)\times \sC^\bullet(A)_{B\mathrm{-hor}}\to\sC^\bullet(A)_{B\mathrm{-hor}},\quad (\eta,\alpha)\mapsto\L_\eta(\alpha).\end{equation}
Note that for $p=1$, $\L$ coincides with the dual of the Bott connection $$\Gamma(B)\times\Gamma(A/B)\to\Gamma(A/B), \quad(\eta,\xi\mod B)\to [\eta,\xi]\mod B.$$ 
By Lemma \ref{cor:final} if $A$ and $B$ are the Lie algebroids of $G$ and of a Lie subgroupoid $K\subset G$ respectively, then \eqref{eq:ad} are the infinitesimal representations of the $K$-actions \eqref{eq:action_K}. The $B$-invariant subspace is defined as \[\sC(A)_{B\mathrm{-inv}}:=\cap_{\eta\in\Gamma(B)}\ker\L_\eta.\]
Since $\L_\eta=\d_{CE}\iota_\eta+\iota_\eta\d_{CE}$, it follows that the $B$-invariant is a subcomplex of $\big(\sC(A),\d_{CE}\big)$. The same identity shows that the following subspace is also a subcomplex:

\begin{definition}The {\it Lie algebroid complex relative to $B$} is the $B$-basic subcomplex defined as the intersection of the $B$-horizontal and the $B$-invariant subspaces
\[\sC(A)_{B\mathrm{-basic}}:=\{\alpha\in \sC(A)\mid \iota_\eta\alpha=0\mathrm{ \ and\ } \L_\eta(\alpha)=0\mathrm{ \ for\  all\ } \eta\in \Gamma(B)\}.\]
The {\em cohomology of $A$ relative to $B$} is the cohomology of $\big(\sC(A)_{B\mathrm{-basic}},\d_{CE}\big)$; we denote it by $H^\bullet(A,B)$. If $B$ is clear from the context, we refer to this cohomology simply as the {\em relative cohomology of $A$}.
\end{definition}

\begin{remark}\label{rmk:wedge}
The usual wedge product turns $\big(\sC(A)_{K-\mathrm{basic}},\d_{CE}\big)\subset \big(\sC(A),\d_{CE}\big)$ into a DG subalgebra and $H^\bullet(A,K)$ (and $H^\bullet(A)$) into a graded algebra. Analogous result holds for $H^\bullet(A,B)$.\end{remark}

\begin{remark}\label{rmk:standard_basic_algebra} For Lie algebras, our definition for the Lie algebroid complex of $\g$ relative to a Lie subalgebra $\h$ coincides with the standard basic subcomplex of the Chevalley-Eilenberg complex $\sC(\g)$. More precisely, the Lie algebroid complex relative to $\h$ coincides with the $\h$-basic subcomplex $\sC(\g)_{\h-\mathrm{basic}}$  of  $\sC(\g)$ defined using the coadjoint action of $\h$ on $\sC(\g)$. Similarly, when $\g$ is the Lie algebra of a Lie group $G$, and $K\subset G$ is a Lie subgroup, our $K$-basic subcomplex coincides with the $K$-basic subcomplex $\sC(\g)_{K-\mathrm{basic}}$ of $\sC(\g)$ defined using the coadjoint action of $K$ on $\sC(\g)$.
\end{remark}

In the case where $A$ and $B$ are the Lie algebroids of the Lie groupoid $G$ and a Lie subgroupoid $K\subset G$ respectively, the following result explains the relation between the two relative Lie algebroid complexes.
 
 \begin{proposition}\label{prop:subcomplex} Let $K$ be a Lie subgroupoid of $G$ and denote by $B$ and $A$ their Lie algebroids, respectively. Then 
$ \sC(A)_{K\mathrm{-basic}}$
is a subcomplex of $\big(\sC(A)_{B\mathrm{-basic}},\d_{CE}\big)$. Moreover, if $K$ has connected $\tz$-fibers, then they are the same.
\end{proposition}

\begin{proof} 
By Lemma \ref{cor:final} and \eqref{eq:Bis_d}, the inclusion $\sC(A)_{K\mathrm{-basic}}\subset \sC(A)_{B\mathrm{-basic}} $ is a subcomplex. 
To see that $\tz$-connectedness implies that they are the same, let $\alpha\in \sC(A)_{B\mathrm{-basic}}$ and set
\[e\colon K\to \wedge^\bullet(A/B)^*,\ \ k\mapsto k\cdot\alpha_{\sz(k)}.\]
Considering $\tz\colon K\to M$ as a bundle, $e$ defines a bundle map covering the identity of $M$.
Note that for any $m\in M\subset K$ and $\eta\in \Gamma(B)$,
\[d_me(\eta)=\frac{d}{d\epsilon}\mid_{\epsilon=0}{\varphi_{\eta}^{\epsilon}(m)}\cdot(\alpha)=\L_\eta(\alpha)_m=0.\]
For any other $k\in K$, since $\varphi_{\eta^L}^{\epsilon}(k)=k\cdot \varphi_{\eta}^{\epsilon}(\sz(k))$,
\[d_ke(\eta^L)=\frac{d}{d\epsilon}\mid_{\epsilon=0}{\varphi_{\eta^L}^{\epsilon}(k)}\cdot\alpha=\frac{d}{d\epsilon}\mid_{\epsilon=0}k\cdot\big({\varphi_{\eta}^{\epsilon}(\sz(k))}\cdot\alpha\big)=k\cdot\big(\L_\eta(\alpha)_{\sz(k)}\big)=0.\] 
 This implies that $de=0$ on $\ker(d\tz)$ and therefore, by connectedness of the $\tz$-fibers, it is constant along the fibers, i.e.\ for $k\in K,$
 $$k\cdot\alpha_{\sz(k)}=\tz(k)\cdot\alpha_{\tz(k)}=\alpha_{\tz(k)}.$$\end{proof}

\section{The transitive case}\label{transitive}

In this section, we construct extension maps from the basic subcomplexes of the isotropy Lie groups and isotropy Lie algebras (see Remark \ref{rmk:standard_basic_algebra}), to the relative cochain complexes of transitive Lie groupoids and Lie algebroids respectively. The outcome is that in the transitive case, the relative cochain complexes of Lie groupoids and Lie algebroids are naturally isomorphic to the cochain basic subcomplexes of their isotropy Lie groups and Lie algebras respectively (see Theorems \ref{thm:ext} and \ref{thm:ext_2}).

Throughout this section, let $G$ be a Lie groupoid over $M$ and let $K\subset G$ be a {\it transitive} Lie subgroupoid, i.e., it has only one orbit: $\tz(\sz^{-1}(m))=M$, for any $m\in M$. Hence, $G$ is also transitive.

\subsection{Lie groupoid cohomology relative to a transitive subgroupoid}\label{sec:transitive_groupoid}

Fix an element  $z\in M$ and let $G_z:= \sz^{-1}(z)\cap \tz^{-1}(z)$, $K_z:=G_z\cap K$,  be the isotropy groups of $G$ and $K$, respectively. Because $K$ is transitive, $\tz\colon P\to M$, $P:=\sz^{-1}(z)\cap K$, is a principal $K_z$-bundle with right action given by right multiplication. Consider the principal $K_z^{p+1}$-bundle given by 
\[(\tz,\ldots,\tz)\colon P^{p+1}\to M^{p+1},\]
and pull it back to $B_pG$ via the map $(g_1,\ldots,g_p)\mapsto(m_0,\ldots,m_p)$:
\[B_pG\times_{M^{p+1}}P^{p+1}\to B_pG.\]
Any function $f\in  \sC^p(G_z)^{K_z}$ defines a smooth function $F$ on $B_pG\times_{M^{p+1}}P^{p+1}$ by the formula
\[F(g_1,\ldots,g_p;k_0,\ldots,k_n)\coloneqq f(k_0^{-1}g_1k_1,\ldots, k_{p-1}^{-1}g_pk_p).\]
In fact, $F$ is $K_z^{p+1}$-invariant as $f$ is invariant w.r.t.\ the $p+1$ commuting $K_z$-actions \eqref{eq:actions}. Hence, $F$ is the pullback of a smooth function of the base $B_pG$, which we call $\mathrm{ext}(f)$:

\begin{definition}\label{def:ext}
    Fix $z\in M$ and let $f\in  \sC^p(G_z)^{K_z}$. We define $\mathrm{ext}(f)\in\sC^p(G)$ by
    \begin{equation}\label{eq:ext}
    \mathrm{ext}(f)(g_1,\ldots,g_p)\coloneqq f(k_0^{-1}g_1k_1,\ldots, k_{p-1}^{-1}g_pk_p), \ \ (g_1,\ldots, g_p)\in B_pG;
    \end{equation}
  here $k_0,\ldots,k_n\in K$ are {\em any} elements satisfying $\sz(k_i)=z$ and $\tz(k_i)=m_i$ for all $i=0,\ldots,p$.
\end{definition}
To illustrate the usefulness of Definition \ref{def:ext}, we state the following result:
  
  \begin{theorem}\label{thm:ext} Suppose $K$ is a transitive Lie subgroupoid of the Lie groupoid $G$. Fix $z\in M$. For any $p$ and any $f\in  \sC^p(G_z)^{K_z}$, the smooth function $\mathrm{ext}(f)$ is an element of $\sC^p(G)^{K}$. Moreover, the extension operator 
    \[\mathrm{ext}\colon \sC^\bullet(G_z)^{K_z}\to \sC^\bullet(G)^{K}\]
is an isomorphism of complexes, which respects the cup product (see Remark \ref{rmk:cup}). Finally, the restriction map
       \[r\colon \sC^\bullet(G)^{K}\to \sC^\bullet(G_z)^{K_z}\]
       is a cochain map inverse to $\mathrm{ext}$, which also respects the cup product. Analogous results hold for $\widetilde{\sC}^\bullet(G_z)^{K_z}$ and $\widetilde{\sC}^\bullet(G)^{K}$.
  \end{theorem}
  
  \begin{proof} 
  All the claims can be checked by a direct calculation. To give an example here, we prove that $\mathrm{ext}(f)$ is an element of $\sC^p(G)^{K}$. Let $(g_1,\ldots,g_p)\in B_pG$ and let $a_0,\ldots,a_p\in K$ be any elements so that $\tz(a_i)=m_i$ for all $i=0,\ldots,p$. Then,
  \begin{align}\label{eq:comp_equiv}
  \mathrm{ext}(f)(a_0^{-1}g_1a_1,\ldots,a_{p-1}^{-1}g_pa_p)&=f(k_0^{-1}(a_0^{-1}g_1a_1)k_1,\ldots, k_{p-1}^{-1}(a_{p-1}^{-1}g_pa_p)k_p)
  \end{align}
  where $k_0,\ldots,k_p\in K$ are any elements so that $\sz(k_i)=z$ for all $i=1,\ldots,p$, and $\tz(k_0)=\tz(a_0^{-1}g_1a_1)=\sz(a_0)$, $\tz(k_i)=\sz(a_{i-1}^{-1}g_ia_i)=\sz(a_i)$ for all $i=1,\ldots,p$. But the r.h.s.\ of equation \eqref{eq:comp_equiv} is equal to
  \begin{align*}
      f((a_0k_0)^{-1}g_1(a_1k_1),\ldots, (a_{p-1}k_{p-1})^{-1}g_p(a_pk_p))=\mathrm{ext}(f)(g_1,\dots, g_p).
\end{align*}
The last equality follows by the definition \eqref{eq:ext} of $\mathrm{ext}(f)$, as $a_0k_0,\ldots, a_pk_p\in K$ are elements satisfying that $\sz(a_ik_i)=\sz(k_i)=z$ and $\tz(a_ik_i)=\tz(a_i)=m_i$ for all $i=0,\ldots,p$. This proves that $\mathrm{ext}(f)$ is invariant w.r.t.\ the $p+1$ $K$-actions \eqref{eq:actions}, i.e.\ $\mathrm{ext}(f)$ belongs to $\sC^p(G)^{K}$.
\end{proof}  

\subsection{Lie algebroid cohomology relative to a transitive subgroupoid}\label{sec:transitive_algebroid}
Let $A$ be the Lie algebroid of $G$, and $B$ the Lie algebroid the subgroupoid $K\subset G$. As $G$ and $K$ are transitive, the anchor map $\a$ of $A$, and its restriction to $B$, is point-wise surjective onto $TM$. Therefore, the kernels \[\g:=\ker\a\subset A,\ \ \h:=B\cap\ker\a\subset B\] are bundles of Lie algebras, fitting into the following exact sequences of Lie algebroids over $M$
\begin{eqnarray*}
0\longrightarrow\g\longrightarrow A\overset{\a}{\longrightarrow}TM\longrightarrow0 \ \ \ \textrm{and}\ \ \ 
0\longrightarrow\h\longrightarrow B\overset{\a}{\longrightarrow}TM\longrightarrow0.
\end{eqnarray*}
Hence, the inclusion $\g\subset A$ gives the identification
\[\g/\h\simeq A/B,\]
and therefore, the restriction map $\big(\sC(A),\d_{CE}\big)\to \big(\sC(\g),\d_{CE}\big) $ induces a natural identification of the $B$-horizontal subspace of $\sC(A)$ with the $\h$-horizontal subspace of $\sC(\g)$: 
\begin{equation}\label{eq:hor_2}\sC(A)_{B\mathrm{-hor}}\simeq\sC(\g)_{\h\mathrm{-hor}}.\end{equation}

We also have a $K$-basic subcomplex of $\sC(\g)$ defined as follows. The restriction to $J^1K$ of the representation \eqref{eq:jet_action} on $A$, sends $\g$ to $\g$, and $\h$ to $\h$, hence $\g/\h$ becomes a representation of $J^1K$. As in the digression presented at the beginning of Subsection \ref{subsection:3.2}, this action actually comes from a representation of $K$ on $\g/\h$. By taking the dual representation on $(\g/\h)^*$ and its exterior powers, we obtain actions of $K$ on $\wedge^\bullet(\g/\h)^*$ analogous to the actions \eqref {eq:action_K}. We define the {\it $K$-basic subcomplex} of $\sC(\g)$ as 
$$\sC(\g)_{K-\mathrm{basic}}=\{\alpha\in\sC(\g)_{\h\mathrm{-hor}}\mid k\cdot \alpha_{\sz(g)}=\alpha_{\tz(g)}, \ k\in K\}.$$
From the construction it is clear that the natural identification \eqref{eq:hor_2} restricts to their $K$-basic subcomplexes:
\begin{equation}\label{eq:identification}
\sC(A)_{K\mathrm{-basic}}\simeq\sC(\g)_{K-\mathrm{basic}}.
\end{equation}\\

The Lie algebras of the isotropy groups $K_z$ and $G_z$ are $\h_{z}$ and $\g_z$ respectively. Given\footnote{By Remark \ref{rmk:standard_basic_algebra} , this is the $K$-basic subcomplex (w.r.t.\ coadjoint action of $K$) of the Chevalley-Eilenberg complex of the Lie algebra $\g_z$.} $\alpha\in \sC(\g_z)_{K_z\mathrm{-basic}}$, we define $\mathrm{ext}(\alpha)\in \sC(\g)_{\k-\mathrm{hor}}$ as 
\begin{equation}\label{eq:ext_algebroid}
\mathrm{ext}(\alpha)_m:=k\cdot\alpha,\quad m\in M,
\end{equation}
where $k\in K$ is {\em any} element with the property that $\sz(k)=z$ and $\tz(k)=m$. 
Once again, in terms of bisections,
\begin{equation}\label{eq:ext_bisection}
    \mathrm{ext}(\alpha)_m=C_b^*(\alpha),
    \end{equation}
where $b$ is any local bisection of $K$ so that $\sz(b(m))=z$, and $C_b$ is defined by \eqref{eq:conjugation}.
Among other things, the next result shows that $\mathrm{ext}(\alpha)$ is a well-defined smooth section.

\begin{proposition}\label{prop:ext} For $\mathrm{ext}\colon \sC^\bullet(\g_z)_{K_z\mathrm{-basic}}\to \sC^\bullet(\g)_{\k-\mathrm{hor}}$ as above, we have the following:
\begin{enumerate}[leftmargin=*]
\item $\mathrm{ext}$ is well-defined, i.e.,\ $\mathrm{ext}(\alpha)_m$ does not depend on the choice of $k\in K$ as above,
\item $\mathrm{ext}$ is a smooth section of $\wedge^\bullet(\g/\h)^*$,
\item the image of $\mathrm{ext}$ is contained in $\sC(\g)_{K-\mathrm{basic}}$, and
\item $\mathrm{ext}$ is a map of complexes, i.e.,\ $\mathrm{ext}\circ\d_{CE}=\d_{CE}\circ \mathrm{ext}.$
\end{enumerate}
\end{proposition}

\begin{proof}
First, we prove parts (a) and (b). Let $P:=\sz^{-1}(z)\cap K$ by the $K_z$-principal bundle $\tz\colon P\to M$ with right action given by right multiplication. For a section $\alpha\in \sC^\bullet(\g_z)_{K_z\mathrm{-basic}}$, define the smooth function
\[F\colon P\to \wedge^p(\g/\h)^*,\quad k\mapsto F(k)=k\cdot \alpha\in \wedge^p(\g_{\tz(k)}/\h_{\tz(k)})^*. \]
It is clear that $F$ takes values in $\wedge^p(\g/\h)^*$: Indeed, as the differential of $C_b$ sends $\k$ to $\k$, then for any $\xi\in\k_{\tz(k)}$,
$$(k\cdot \alpha)(\xi)=C_b^*(\alpha)(\xi)=\alpha(d_{\t(k)}C_b(\xi))=0.$$
Moreover, as $\alpha$ is $K_z$-invariant, then $F$ is $K_z$-invariant as well: Indeed, if $h\in K_z$, then 
\begin{align*}
F(kh)=(kh)\cdot\alpha=k\cdot (h\cdot \alpha)=k\cdot \alpha=F(k).
\end{align*}
This means that $F$ is the pullback $\tz^*(f)$ of a smooth function $f\colon M\to \wedge^p(\g/\h)^*$, which is precisely $\mathrm{ext}(\alpha)$. The fact that $\mathrm{ext}(\alpha)$ is a section follows directly from the definition of $F$. 

To see (c), we fix $h\in K$. Let $k\in K$ be any element so that $\sz(k)=z$ and $\tz(k)=\sz(h)$. Then
\begin{align*}
    h\cdot\mathrm{ext}(\alpha)_{\sz(h)}=h\cdot (k\cdot \alpha)=hk\cdot \alpha=\mathrm{ext}(\alpha)_{\tz(hk)}=\mathrm{ext}(\alpha)_{\tz(h)}. 
\end{align*}

For (d), we use the description of $\mathrm{ext}(\alpha)$ in terms of bisections of $K$ (see \eqref{eq:ext_bisection}). Item (d) then follows by a straightforward computation by noting that $DC_{b}:\g|_{\mathrm{Dom}(b)}\to \g|_{\mathrm{Im}(\sz\circ b)}$ is a Lie algebroid morphism and that both $\g$ and $\g_z$ have trivial anchor.
\end{proof}

By Proposition \ref{prop:ext}, and as a consequence of \eqref{eq:identification}, we obtain the following:

\begin{theorem}\label{thm:ext_2}
Suppose $K$ is a transitive Lie subgroupoid of the Lie groupoid $G$. The extension operator 
    \[\mathrm{ext}\colon \sC^\bullet(\g_z)_{K_z\mathrm{-basic}}\to \sC^\bullet(A)_{K\mathrm{-basic}}\]
    defined by \eqref{eq:ext_algebroid} is an isomorphism of complexes, which respects the wedge product. The restriction map $$r\colon\sC^\bullet(A)_{K\mathrm{-basic}}\to\sC^\bullet(\g_z)_{K_z\mathrm{-basic}}$$ is a cochain map inverse to $\mathrm{ext}$, which also respects the wedge product.
\end{theorem}

\section{Van Est maps}\label{section:van-est}

In this section, we study the van Est map in the context of relative cohomology and the case when there is a right inverse, called the van Est integration map.
We define a tubular structure and a Cartan decomposition for a pair $(G,K)$ consisting of a Lie groupoid $G$ and  Lie subgroupoid $K$, and make use of this notion to define a van Est integration map (see Theorem \ref{prop:rightinversegk}). In Sections \ref{section:differentiation_proper} and \ref{section:integration_proper}, we specialize to the case in which $K$ is proper, while in Section \ref{section:properandtransitive} we additionally impose the condition that $K$ is transitive. We conclude with the example of $GL_n(\C)$ in Section \ref{section:example}; this application becomes particularly relevant for the definition of characteristic classes in Section \ref{section:cc}.

\subsection{Tubular structures and Cartan decomposition for $(G,K)$}
  
  Let $K$ be a closed embedded Lie subgroupoid of $G$, e.g.\ any proper Lie subgroupoid $K$. Recall from \cite{Macrun} that in this case, the space of left cosets $G/K$ is a Hausdorff manifold and the projection $G\to G/K$ is a submersion. Denote by $A$ and $B$ the Lie algebroids of $G$ and $K$ respectively.
  
  Denote by $\bar{\tz}:G/K\to M$ the surjective submersion induced by the target map. Note that $\bar{\uz}\colon M\to G/K$, the map induced by the unit map of $G$, is a canonical section of $\bar{\tz}$. Along this unit section $M\subset G/K$, the $\bar{\tz}$-fibers give a canonical decomposition of the tangent space $T(G/K)|_M=TM\oplus A/B$. In particular, the normal bundle of $M$ in $G/K$ is naturally identified with $A/B$. This motivates the following definition:
  
 \begin{definition}\label{def:tubular} A \emph{(global) tubular structure} of $(G,K)$ is a diffeomorphism of bundles over $M$
   \begin{equation*}\xymatrix{  
  A/B  \ar[r]^-{\phi } \ar[d]_q&   G/K \ar[d]^{\bar{\tz}}  \\
M \ar@{=}[r]& M 
 }
 \end{equation*}
such that $\phi(0_m)=\bar{\uz}(m)$ for $m\in M$. 
 \end{definition}

\begin{definition}\label{def:Cartan}Let
   \begin{equation*}\xymatrix{  
  A/B  \ar[r]^-{\psi } \ar[d]_q&   G \ar[d]^{\tz}  \\
M \ar@{=}[r]& M 
 }
\end{equation*}
 be a bundle map taking values in the isotropy groups of $G$ and sending the zero section to $K$. If 
\begin{equation}\label{eq:diagram-cartan}
\Psi\colon A/B\tensor[_q]{\times}{_\tz}K\to G,\quad  \Psi (\xi,k)= \psi(\xi) \cdot k
\end{equation}
is a $K$-equivariant diffeomorphism w.r.t.\ the diagonal action and the left multiplication, then $\Psi$ is called a {\it (global) Cartan decomposition} of $(G,K)$.

\end{definition}  

\begin{remark}\label{rmk:local} In the literature, there are notions that are analogous to those of Definitions \ref{def:tubular} and \ref{def:Cartan} and are {\em local}, i.e., they are only defined near $M$ (see \cite{locales,Eckhard_2}). This explains why ``(global)'' appears in both definitions. In the spirit of \cite{Eckhard_2}, some of the constructions and results below hold in this more general context by replacing (relative) groupoid cohomology by the germ version. Since our emphasis is on global results, from now on we drop ``(global)'' from the terminology.  
\end{remark}


\begin{remark}\label{rmk:Cartan} A Cartan decomposition of $(G,K)$ induces a tubular structure for $G/K$. Indeed, the map
\begin{equation}\label{eq:tubular}\phi\colon A/B\to G/K,\ \  \xi\mapsto \psi(\xi)K\end{equation}
is the required diffeomorphism.
\end{remark}

\begin{remark} 
The fact that \eqref{eq:tubular} is a bijective map imposes some restrictions for the pair $(G,K)$. In particular, $K$ is forced to have the same orbits as $G$ and, denoting by $G(m,m')=\sz^{-1}(m')\cap\tz^{-1}(m)$, for $m,m'\in M$ (and similarly for $K$),
\[G(m,m')\cdot K(m',m)\subset \psi((A/B)_m).\]
\end{remark}

The following example explains the relation that our Definitions \ref{def:tubular} and \ref{def:Cartan} have with standard results in Lie theory (see e.g.\ \cite[Theorem 6.31]{knapp} and \cite[Theorem 32.5]{stroppel}). 

\begin{example}\label{exm:semisimple}(See \cite[Chapter VII]{knapp}) Let $G$ be a Lie group with finitely many connected components and Lie algebra $\g$, and let $K\subset G$ be a maximal compact subgroup with Lie algebra $\mathfrak{k}$. Then the diffeomorphism $G/K\simeq \g/\mathfrak{k}$ (see e.g.\ \cite[Theorem 32.5]{stroppel}) is a tubular structure of $(G,K)$.    

Moreover, suppose that $G$ is reductive, i.e.\ there exists a Lie algebra involution $\theta$ of $\g$, and a non-degenerate, $\mathrm{Ad}(G)$- and $\theta$-invariant bilinear form $B$ on $\g$ such that 
\begin{enumerate}
    \item $\g$ is a reductive Lie algebra,
    \item the decomposition of $\g$ into $+1$ and $-1$ eigenspaces under $\theta$ is $\g=\mathfrak{k}\oplus\mathfrak{p}$, where $\k$ is the Lie algebra of $K$,
    \item $\k$ and $\mathfrak{p}$ are orthogonal under $B$, $B$ is positive definite on $\mathfrak{p}$ and negative definite on $\k$,
    \item for any $g \in G$, the automorphism $\mathrm{Ad}(g)$ of $\g^{\C}$ is inner, and
    \item the map \begin{equation}\label{eq:CartanDecomposition}\mathfrak{p}\times K\to G,\ \ (\xi,k)\mapsto \exp(\xi)\cdot k,\end{equation}
    called the {\it global Cartan decomposition}, is a diffeomorphism.
\end{enumerate}
The map \eqref{eq:CartanDecomposition} is a Cartan decomposition of the pair $(G,K)$; here, $\g/\mathfrak{k}$ is canonically identified with $\mathfrak{p}.$ Examples of reductive Lie groups are connected, semisimple Lie groups $G$ with finite center, $K$ the maximal compact subroup, $\theta$ a Cartan involution and $B$ the Killing form. Other examples include connected closed linear groups of real or complex matrices closed under conjugate transpose, $\theta$ its negative conjugate transpose, $K$ its intersection with the unitary group and $B(X, Y)=\mathrm{Re}\ \mathrm{Tr}(X,Y)$.
\end{example}

\begin{example}\label{exm:transitive} If $G$ is a transitive groupoid whose isotropies have finitely many connected components, then
there exists a proper subgroupoid $K$ for which $(G,K)$ admists a tubular structure. 

To see this, without loss of
generality we assume that $G$ is the Gauge groupoid $\mathrm{Gauge}(P):=(P\times P)/G_m$ of
the principal $G_m$-bundle $P\to M$ given by the source fiber of a specified point $m\in M$ and where $G_m$ is the isotropy group at $m$. 
Consider now the maximal compact subgroup $K_m$ of $G_m$ and choose a reduction of the structural group to $K_m$, i.e.\ a principal $K_m$-subbundle $Q\to M$. Such a $Q$ always exists as this is guaranteed by the existence of a section $s$ of the bundle $P\tensor[]{\times}{_{G_m}}(G_m/K_m)\to M$ with contractible fibers $G_m/K_m$ (indeed, out of $s$ the fibers of $Q$ are defined by $Q_z=s'(z)\cdot K,\ z\in M$, where $s'(z)\in P$ is any element with the property that $s(z)=[s'(z), eK]$). 
Then, acting by $G_m$ on the subset $Q\subset P$, we have the isomorphism of principal $G_m$-bundles
\[Q\tensor[]{\times}{_{K_m}} G_m\overset{\simeq}{\longrightarrow}P,\ \ [p,g]\mapsto p\cdot g, \]
and therefore,
\begin{eqnarray*}(P\times P)/G_m\simeq (Q\times Q)\tensor[]{\times}{_{K_m\times K_m}}(G_m\times G_m)/G_m\simeq(Q\times Q)\tensor[]{\times}{_{K_m\times K_m}}G_m,\\
(Q\times Q)/K_m\simeq (Q_m\times Q_m)\tensor[]{\times}{_{K_m\times K_m}}(K_m\times K_m)/K_m\simeq(Q\times Q)\tensor[]{\times}{_{K_m\times K_m}}K_m,
\end{eqnarray*}
where on the right hand side the right action of $K_m\times K_m$ on $G_m$ (and $K_m$) is given by $(a,b)\cdot g=a^{-1}gb$. Hence, $K:=\mathrm{Gauge}(Q)\subset \mathrm{Gauge}(P)=G$ is a Lie subgroupoid which is proper as the fiber of $Q\to M$ is compact, and with the property that its isotropy group $K_m$ is the maximal compact subgroup of the isotropy group $G_m$ of $G$. 

The bundle of left cosets $\bar\tz\colon G/K\to M$ is diffeomorphic as a bundle to the bundle of left cosets $\bar\tz\colon\bar{G}/\bar{K}\to M$, where $\bar{G}$ and $\bar{K}$ are the Lie subgroupoids of $G$ and $K$ given by the union of the isotropy groups. Analogously, the quotient $A/B$ of the Lie algebroids of $G$ and $K$ is isomorphic as a vector bundle to the quotient $\g/\k$ of the kernels of the anchors $\g=\ker(\a)\subset A$, $\k=\ker(\a)\subset B$. As $\g$ and $\k$ are locally trivial Lie algebra bundles with fibers $\g_m$, $\k_m$ -- the Lie algebras of $G_m$ and $K_m$, respectively -- then the isomorphism $\g_m/\mathfrak{k}_m\simeq G_m/K_m$ extends to a tubular structure
\[A/B\simeq\g/\k\to \bar{G}/\bar{K}\simeq G/K.\]

Similarly, a transitive Lie groupoid whose isotropies have finitely many connected components and are reductive admits a proper subgroupoid with a Cartan decomposition. In this case, the Cartan decomposition for the isotropy Lie groups \eqref{eq:CartanDecomposition} extends to a diffeomorphism
$ \bar\Psi\colon A/B\tensor[_q]{\times}{_\tz} \bar{K} \to  \bar{G}$
that extends naturally to a Cartan decomposition 
\[ A/B\tensor[_q]{\times}{_\tz} K \overset{\Psi}{\longrightarrow}  G. \]
%
%
 
\end{example}


\subsection{Van Est differentiation map for relative complexes}
 The van Est differentiation map is a cochain map
\begin{equation*}\label{eq:vanEst_differentiation}VE_G\colon\sC^\bullet(G)\to\sC^\bullet(A).\end{equation*}
In Appendix \ref{appendix}, we recall from \cite{Eckhard_2} how the van Est differentiation map is obtained using the Perturbation Lemma. 
Explicitly, the commuting left $G$-actions on $B_pG$ obtained from \eqref{eq:actions} in the standard fashion have generating vector fields $\xi^{(i)}\in\mathfrak{X}(B_pG),\ \xi\in\Gamma(A),\ i=0,\ldots,p$.

\begin{theorem}[\cite{Eckhard,Eckhard_2}]
The map $VE_G$ is given by the formula
\begin{equation}\label{eq:vanEst_expression}VE_G(f)(\xi_1,\ldots,\xi_p)=\sum_{s\in S_p}\mathrm{sign}(s)\ \L_{\xi_{s(1)}^{(1)}}\cdots\L_{\xi_{s(p)}^{(p)}}(f)|_M\end{equation}
for $f\in \sC^p(G)$ and $\xi_1,\ldots,\xi_p\in\Gamma(A)$. Here the sum is over the permutation group $S_p$ and $M$ is regarded as a submanifold of $B_pG$ consisting of constant trivial $p$-arrows.
\end{theorem}

\begin{remark}
    The formula \eqref{eq:vanEst_expression} is closely related to the Weinstein-Xu formula for the van Est map on normalized cochains in \cite{Weinstein_Xu} (see \cite[Remark 4.5]{Eckhard_2} for more details).
\end{remark}

For the relative subcomplexes, we obtain the following result:

\begin{theorem}\label{cor:VE_relative}
Let $K$ be a Lie subgroupoid of $G$. Then $VE_G$ restricts to a cochain map of the $K$-basic subcomplexes
\[VE_G\colon\sC^\bullet(G)^K\to\sC^\bullet(A)_{K\mathrm{-basic}}.\] 
Moreover, when restricted to $\widetilde{\sC}^\bullet(G)^K$, $VE_G$ is an algebra map, where the products are the ones in Remarks \ref{rmk:cup} and \ref{rmk:wedge}.
\end{theorem}

\begin{proof}
Let $f\in \sC^p(G)^K$. To see that $\iota_\eta (VE_G(f))=0$ for any $\eta\in\Gamma(B)$, note that the Lie derivatives on expression \eqref{eq:vanEst_expression} commute as the $G$-actions commute with each other. Hence, 
$$VE_G(f)(\xi_1,\ldots,\xi_{p-1},\eta)$$
is computed as the sum (with appropriate signs) of expressions of the form
\[\L_{\xi_{s(1)}^{(1)}}\cdots\L_{\xi_{s(p-1)}^{(p-1)}}\L_{\eta^{(i)}}(f)|_M.\]
But $f$ is invariant w.r.t.\ the $p+1$ $K$-actions \eqref{eq:actions}, and so $\L_{\eta^{(i)}}(f)=0$, $i=1,\ldots, p$. 

To see that $VE_G(f)$ is $K$-invariant, let $b\in \Bis(K)$ and $m\in M$. Then
\begin{equation}\label{eq:VE}
(b\cdot VE_G(f))(\xi_1,\ldots,\xi_p)(m)=\sum_{s\in S_p}\mathrm{sign}(s)\ \L_{\big(C_{b,*}(\xi_{s(1)})\big)^{(1)}}\cdots\L_{\big(C_{b,*}(\xi_{s(p)})\big)^{(p)}}(f)|_{(C_b(m),\ldots C_b(m))}.
\end{equation}
The generating vector fields $\xi^{(i)}\in\mathfrak{X}(B_pG),\ \xi\in\Gamma(A)$ are of the form
\begin{equation*}\label{eq:infinitesimal_generators} 
\xi^{(i)}_{g_1,\ldots,g_p}=\begin{cases}
(d\iz(\xi)^R_{g_1},0_{g_2},\ldots,0_{g_p}) & i=0,\\
(0_{g_1},\ldots,\xi^L_{g_i},d\iz(\xi)^R_{ g_{i+1}},\ldots,0_{g_p})&0<i<p,\\
(0_{g_1},\ldots,0_{g_{p-1}},\xi^L_{g_p}) &i=p,
\end{cases}
\end{equation*}
where $\xi^L,d\iz(\xi)^R\in\mathfrak{X}(G)$ denote the left and right invariant vector fields associated to $\xi$, respectively. From this description, it is a straightforward computation to see for any $b\in\Bis(G)$, the pushforward of $\xi^{(i)}$ by the diffeomorphism $C^p_b\colon B_pG\to B_pG,\ C^p_b=C_b\times\cdots\times C_b$, is such that 
\[C^p_{b,*}\big(\xi^{(i)}\big)=\big(C_{b,*}(\xi)\big)^{(i)}.\]
Since $f$ is $K$-invariant, $(C_b^{p,-1})^*(f)=f$ for any $b\in \Bis(G)$, where $C_b^{p,-1}$ denotes the function inverse to $C_b^{p}$. Hence, we compute
\begin{multline*}\label{eq:mult}
\L_{\big(C_{b,*}(\xi_{s(1)})\big)^{(1)}}\cdots\L_{\big(C_{b,*}(\xi_{s(p)})\big)^{(p)}}(f)=\L_{\big(C_{b,*}(\xi_{s(1)})\big)^{(1)}}\cdots\L_{\big(C_{b,*}(\xi_{s(p)})\big)^{(p)}}((C_b^{p,-1})^*(f))\\ 
=\L_{C^p_{b,*}\big(\xi_{s(1)}^{(1)}\big)}\cdots\L_{C^p_{b,*}\big(\xi_{s(p)}^{(p)}\big)}((C_b^{p,-1})^*(f))= (C_b^{p,-1})^*\big(\L_{\xi_{s(1)}^{(1)}}\cdots\L_{\xi_{s(p)}^{(p)}}(f)\big).
\end{multline*}
In the last line, we use recursively that $\L_{C^p_{b,*}(\xi^{(i)})}\big((C^{p,-1}_b)^*(F)\big)=(C^{p,-1}_b)^*\big(\L_{\xi^{(i)}}(F)\big)$ for any function $F\in C^\infty (B_pG)$. Hence, 
\begin{equation}\label{eq:mult-}
\L_{\big(C_{b,*}(\xi_{s(1)})\big)^{(1)}}\cdots\L_{\big(C_{b,*}(\xi_{s(p)})\big)^{(p)}}(f)|_{(C_b(m),\ldots C_b(m))}= \L_{\xi_{s(1)}^{(1)}}\cdots\L_{\xi_{s(p)}^{(p)}}(f)|_{(m,\ldots ,m)}
\end{equation}
Putting together \eqref{eq:VE} and \eqref{eq:mult-} we obtain that $VE_G(f)\in \sC^p(A)^K$. 

Finally, to see that when restricted to $\widetilde{\sC}^\bullet(G)^K$, $VE_G$ becomes an algebra map, recall the following: When considering the normalized subcomplex $\widetilde\sC^\bullet(G)\subset \sC^\bullet(G)$, the restriction $VE_G:\widetilde\sC^\bullet(G)\to \sC^\bullet(A)$  is an algebra map \cite{Crainic:vanEst,Eckhard}. As $\widetilde{\sC}^\bullet(G)^{K}\subset\widetilde \sC^\bullet(G)$ and $\sC^\bullet(A)_{K-\mathrm{basic}}\subset \sC^\bullet(A)$ are subalgebras, then $VE_G|_{\widetilde{\sC}^\bullet(G)^K}$ is also an algebra map.
\end{proof}

  \subsection{Van Est differentiation maps in the proper case}\label{section:differentiation_proper} 
Let $K\subset G$ be a proper subgroupoid.
 Choose any \emph{left invariant} normalized Haar system $\mu$ on $K$
and denote by $VE_{G/K}$ the composition 
 \begin{equation}\label{eq:vanEst_relative}VE_{G/K}\colon \sC^\bullet(G)\overset{\on{Av}}{\longrightarrow}  \sC^\bullet(G)^{K}\overset{VE_G}{\longrightarrow}\sC^\bullet(A)_{K\mathrm{-basic}}; \end{equation}
here, $\on{Av}$ is the homotopy equivalence given by the averaging operator \eqref{eq:averaging} associated to $\mu$. 

In this section, we use the Perturbation Lemma to construct a cochain map $\sC^\bullet(G) \to \sC^\bullet(A)_{K\mathrm{-basic}}$ and show that it coincides with  $VE_{G/K}$. This allows us to obtain van Est integration maps in Section \ref{section:integration_proper}.
We refer to Appendix \ref{appendix}, where we summarize some of the results and techniques of \cite{Eckhard_2} to obtain $VE_G$ (see \eqref{eq:vanEst_expression}) via the Perturbation Lemma. \\

Let $\sD^{p,q}(G)=\Gamma\big(\pi_p^*(\wedge^q A^*)\big)$ be the auxiliary double complex, where $\pi_p \colon E_p G=B_pG\times_M G\to M, \,\pi_p(g_1,\ldots,g_p;g)=\sz(g)$ (see \eqref{eq:double}). We write 
$\sD^{\bullet,\bullet}(G)_{B-\mathrm{inv}}$
for the subspace of $\sD^{\bullet,\bullet}(G)$ consisting of 
sections that are annihilated by contractions with elements of $B\subset A$ --the Lie algebroid of $K$--. Hence one can think of $\sD^{p,q}(G)_{B-\mathrm{inv}}$ as sections of $\Gamma\big(\pi_p^*(\wedge^q (A/B)^*)\big)$. By a direct calculation, both differentials $\delta$ and $\d$ of the auxiliary double complex restrict to $\sD^{\bullet,\bullet}(G)_{B-\mathrm{inv}}$, i.e.\ $(\sD^{\bullet, \bullet}(G)_{B-\mathrm{inv}}, \delta,\d)$ is a double subcomplex.

The $K$-action on $\wedge^q (A/B)^*$ of \eqref{eq:action_K} induces the following $K$-action on  $\sD^{p, q}(G)_{B-\mathrm{inv}}$:
\begin{equation}\label{eq:principalaction}
(k\cdot\psi)(g_1,\ldots,g_p;g)=k\cdot \psi(g_1,\ldots,g_p;gk).
\end{equation}
We denote by 
\begin{equation}\label{eq:Kbasic} \sD^{\bullet, \bullet}(G)_{K\mathrm{-basic}}\end{equation}
the $K$-invariant subspace of $\sD^{\bullet,\bullet}(G)_{B-\mathrm{inv}}$. By a direct calculation, the differentials $\delta$ and $\d$ commute with this action. Hence, $(\sD^{\bullet, \bullet}(G)_{K\mathrm{-basic}}, \delta,\d)$ is a subcomplex, called the $K$-basic subcomplex.

The $K$-basic subcomplex comes with horizontal and vertical augmentation maps\linebreak
$\si\colon \sC^\bullet(A)_{K\mathrm{-basic}}\to 
\sD^{0,\bullet}(G)_{K\mathrm{-basic}}$ given by the pullback $\pi_0^*$ and by $\sj\colon \sC^\bullet(G)
\to \sD^{\bullet,0}(G)_{K\mathrm{-basic}}$ given in degree $p$ by the pullback map $\kappa_p^*$, where $\kappa_p\colon E_pG\to B_pG$ is the projection. The augmentation maps define cochain maps to the total complex 
\[ \sC^\bullet(G)\stackrel{\sj}{\lra} \on{Tot}^\bullet(\sD(G)_{K\mathrm{-basic}})\stackrel{\si}{\longleftarrow}\sC^\bullet(A)_{K\mathrm{-basic}} .\]
Since the Haar measure $\mu$ is left invariant, the $K$-basic subcomplex also comes with a horizontal homotopy operator $\mathsf{h}_K \colon \sD^{\bullet,\bullet}(G)_{K\mathrm{-basic}}\to \sD^{\bullet-1,\bullet}(G)_{K\mathrm{-basic}}$ given by
\begin{equation}\label{eq:hmap} (\mathsf{h}_K\psi)(g_1,\ldots,g_{p-1};g)=(-1)^p \int_{\tz^{-1}(m)}\ \psi(g_1,\ldots,g_{p-1},ga;a^{-1})\mu_{m}(a);\end{equation}
here $m=\sz(g)$. Indeed, a direct calculation shows that 
$\delta\mathsf{h}_K+\mathsf{h}_K\delta=1-\si\circ \sp_K$, where $\sp_K$ vanishes on elements of bidegree $(p,q)$ with $p>0$, while
\begin{equation}\label{eq:pmap}
 \sp_K(\psi)=\int_{\tz^{-1}(m)} \psi(a^{-1})\mu_{m}(a)
,\end{equation} 
for $\psi\in \Gamma\big(\pi_0^*(\wedge^q A^*)\big)_{K\mathrm{-basic}}= \Gamma\big((\ker d\tz)^*\big)_{K\mathrm{-basic}}$. Using the Perturbation Lemma \ref{lem:perturbed}, we obtain a cochain map
\[\sp_K\circ (1+\d\mathsf{h}_K)^{-1}\circ \sj\colon 
\sC^\bullet(G)\to \sC^\bullet(A)_{K-\on{basic}}.\]

\begin{remark}
For $K=M$, one recovers the homotopy operator  \[ (\mathsf{h}\psi)(g_1,\ldots,g_{p-1};g)=(-1)^p  \psi(g_1,\ldots,g_{p-1},g;m).\] 
and the map $\sp=\uz^*$, discussed in the Appendix \ref{appendix}. In particular, we recover the formula 
$VE_G=\sp\circ (1+\d\mathsf{h})^{-1}\circ \sj$ from Theorem \ref{Thm:appendix}.
\end{remark}

The following form of $VE_{G/K}$ via the Perturbation Lemma will allow us to conclude in Theorem \ref{prop:rightinversegk} that $VE_{G/K}$ becomes a homotopy equivalence in the presence of tubular structures for $(G,K)$.

\begin{proposition}\label{thm:vanest}
For any proper Lie subgroupoid $K\subset G$, the cochain map $\sp_K\circ (1+\d\mathsf{h}_K)^{-1}\circ \sj$ equals 
\[\sp\circ (1+\d\mathsf{h})^{-1}\circ \sj\circ\on{Av}.\]
In particular, $VE_{G/K}=\sp_K\circ (1+\d\mathsf{h}_K)^{-1}\circ \sj$.
\end{proposition}

In preparation for the proof of Proposition \ref{thm:vanest}, in analogy with \eqref{eq:av} for each $p,q \geq 0$ and each $0 \leq i \leq p$, we denote by 
\[\on{Av}^{(i)}_{\sD(G)}\colon \sD^{p,q}(G)\to \sD^{p,q}(G) \]
the induced averaging operator with respect to the $i$-th $K$-action \eqref{eq:actions2}, and by \[\on{Av}_{\sD(G)}=\on{Av}^{(0)}_{\sD(G)}\circ \cdots \circ \on{Av}_{\sD(G)}^{(p)}\]
 the total $K^{p+1}$-averaging operator. It is immediate to see that
\begin{align}\label{eq:avcommutes}\mathsf{h}_K|_{\sD^{p,q}(G)}=\mathsf{h}\circ\on{Av}^{(p)}_{\sD(G)}|_{\sD^{p,q}(G)},\  \ \, \sp_K=\sp\circ \on{Av}^{(0)}_{\sD(G)}
,\end{align}
 \begin{align}\label{eq:avcommutes2}\mathsf{h}\circ\on{Av}^{(p-1)}_{\sD(G)}|_{\sD^{p,q}(G)}=\on{Av}^{(p-1)}_{\sD(G)}\circ\ \mathsf{h}|_{\sD^{p,q}(G)},\  \ \,\on{Av}_{\sD(G)}\circ\ \sj=\sj\circ\on{Av}.\end{align}
We stress that the equality on the left of \eqref{eq:avcommutes} (respectively \eqref{eq:avcommutes2}) holds for the $p$-th (respectively $(p-1)$-th) averaging operator on $\sD^{p,q}(G)$.

\begin{lemma}\label{lemma:av} For all $i$, the averaging operators commute with the vertical differential: 
\[\d\circ\on{Av}^{(i)}_{\sD(G)}=\on{Av}^{(i)}_{\sD(G)}\circ\ \d.\]
\end{lemma}

\begin{proof} By \eqref{eq:dlacomplex}, we have that $\d=-(1)^p\d_{Rh}$ is the foliated de Rham differential on elements of bidegree $(p,q)$ under the isomorphism $\pi_p^* A\cong T_\F (E_pG)$. For a section $\xi\in \Gamma(A)$, the pullback $\pi_p^*(\xi)\in \Gamma(\pi_p^* A)$ is identified with $\hat\xi\in \Gamma(T_\F E_p)$ given by
\[\hat\xi_{(g_1,\ldots,g_p;g)}=(0_{g_1},\ldots,0_{g_p}, \xi^L_{g}).\]
Hence, the flow of $\hat\xi$ is given by
\[\varphi^\epsilon_{\hat\xi}(g_1,\ldots,g_p;g)=\left(g_1,\ldots,g_p;g\varphi^\epsilon_{\xi}(\sz(g)\right)\]
(cf.\ \eqref{eq:left-invariant}). From this expression it is immediate to see that $\varphi^\epsilon_{\hat\xi}$ commutes with the $p+1$ $G$-actions \eqref{eq:actions2} and therefore, for $\psi\in C^\infty (E_pG)$, 
$$\d_{Rh}\left(\on{Av}^{(i)}_{\sD(G)}\psi \right)(\hat\xi)=\on{Av}^{(i)}_{\sD(G)}\left(\d_{Rh}\psi \right)(\hat\xi) \quad i=0,\ldots,p.$$ 
Since $\Gamma(T_\F E_p)$ is generated by sections $\hat\xi$ for $\xi\in\Gamma(A)$, then $\d_{Rh}\circ\on{Av}^{(i)}_{\sD(G)}\psi =\on{Av}^{(i)}_{\sD(G)}\circ\ \d_{Rh}\psi$ for any $\psi\in C^\infty (E_pG)$. Hence, $\d_{Rh}\circ\on{Av}^{(i)}_{\sD(G)}\psi =\on{Av}^{(i)}_{\sD(G)}\circ\ \d_{Rh}\psi$ for $\psi\in \Gamma(\wedge ^qT^*_\F E_p)\simeq\sD^{p,q}(G)$  and any $q\geq0$.
\end{proof}

\begin{proof}[Proof of Proposition \ref{thm:vanest}]
At degree $p$, we compute $\sp_K\circ (1+\d\mathsf{h}_K)^{-1}\circ \sj$:
\begin{align*}
\sp_K\circ (\d \mathsf{h}_K)^p \circ \sj&=
\sp\circ \on{Av}^{(0)}_{\sD(G)} (\d \mathsf{h}_K)^p \circ \sj\\
 &=\sp\circ \on{Av}^{(0)}_{\sD(G)}\circ (\d\mathsf{h}\on{Av}^{(1)}_{\sD(G)} \circ \cdots\circ
\d\mathsf{h}\on{Av}^{(p)}_{\sD(G)})
 \circ \sj\\&=\sp\circ (\d\on{Av}^{(0)}_{\sD(G)}\mathsf{h}\circ\d\on{Av}^{(1)}_{\sD(G)}\mathsf{h} \circ \cdots\circ
\d\on{Av}^{(p-1)}_{\sD(G)}\mathsf{h}\on{Av}^{(p)}_{\sD(G)})
 \circ \sj\\&=\sp\circ (\d\mathsf{h}\on{Av}^{(0)}_{\sD(G)}\circ\ \d\mathsf{h}\on{Av}^{(1)}_{\sD(G)} \circ \cdots\circ
\d\mathsf{h}\on{Av}^{(p-1)}_{\sD(G)}\on{Av}^{(p)}_{\sD(G)})
 \circ \sj\\
 &=\ldots
 \\ &=\sp\circ (\d \mathsf{h})^p \circ \on{Av}_{\sD(G)}\circ\sj\\
 &=\sp\circ (\d \mathsf{h})^p \circ\sj\circ \on{Av},
\end{align*}
where in the second equality we used \eqref{eq:avcommutes}, in the third we applied Lemma \ref{lemma:av}, in the fourth we used the left equation of \eqref{eq:avcommutes2}, and in the last we used the right equation of \eqref{eq:avcommutes2}.
\end{proof}

\subsection{Van Est integration maps in the proper case}\label{section:integration_proper} 

Let $K$ be a proper Lie subgroupoid of $G$. We will use the Perturbation Lemma to obtain van Est integration maps (see  Appendix \ref{appendix} for background on the Perturbation Lemma).
For the discussion, another interpretation of the $K$-basic (double) complex will be convenient (see the case of Lie groups in \cite[Section 3.3]{Eckhard_2}). In this case, we will make use of a tubular structure to obtain a vertical homotopy operator on the double complex and define the van Est integration map $R_{G/K}\colon\sC^\bullet(A)_{K-\mathrm{basic}}\to \sC^\bullet(G)$.\\

The manifold $G/K$ comes equipped with a left $G$-action 
\[a\cdot gK=agK\]
that happens along the map $\bar{\tz}:G/K\to M$. Let $(\Omega^\bullet_\F(G/K)^L,\d_{Rh})$ be the de Rham complex of $G$-invariant $\bar{\tz}$-foliated forms; this can be seen as a subcomplex of $(\Omega^\bullet_\F(G)^L,\d_{Rh})$ via pulling back along the quotient map $G \to G/K$. Hence, by unraveling the isomorphism of differential complexes \eqref{eq:forms} and the definition of $\sC^\bullet(A)_{K\mathrm{-basic}}$, the restriction of \eqref{eq:forms} to $\Omega^\bullet_\F(G/K)^L$ is an isomorphism onto $\sC^\bullet(A)_{K\mathrm{-basic}}$:
  \begin{equation}\label{eq:dcedrh}\big(\sC^\bullet(A)_{K\mathrm{-basic}},\d_{CE}\big)\simeq\big(\Omega^\bullet_\F(G/K)^L,\d_{Rh}\big).\end{equation}
Similarly, elements in the double complex $\sD^{p,q}(G)_{K-\on{basic}}$ as in \eqref{eq:Kbasic} may be identified with functions  
\begin{equation}\label{eq:betamap} \beta\colon B_pG\to \Omega^q_\F(G/K)\end{equation}
such that $\beta(g_1,\ldots,g_p)\in \Omega^q(\bar\tz^{-1}(m))$ for $m=\sz(g_p)$, and smoothly depending on $(g_1,\ldots,g_p)$. (For $p=0$, this is to be interpreted as $\beta\colon M\to \Omega^q_\F(G/K)$ such that $\beta(m)\in \Omega^q(\bar\tz^{-1}(m))$.) In these terms, the two
differentials are 
\begin{equation}\label{eq:dbeta}
 (\d\beta)(g_1,\ldots,g_p)=(-1)^p \d_{Rh} \beta(g_1,\ldots,g_p),
 \end{equation}
 where $\d_{Rh}$ is the de Rham differential, and 
\begin{align} \label{eq:deltabeta}(\delta\beta)(g_1,\ldots,g_{p+1})&=
\beta(g_2,\ldots,g_{p+1})+\sum_{i=1}^p (-1)^i \beta(g_1,\ldots,g_ig_{i+1},\ldots,g_{p+1})
\\&\ \ +(-1)^{p+1} L(g_{p+1})^*\beta(g_1,\ldots,g_p),\nonumber\end{align}
where $L(a)\colon \bar{\tz}^{-1}(\sz(a))\to  \bar{\tz}^{-1}(\tz(a))$ is 
the action of $a\in G$. 
(For $p=0$, this is to be interpreted as $(\delta\beta)(g_1)=\beta-L(g_1)^*\beta$.) 
The horizontal augmentation map $\si$ is the inclusion of the invariant forms \eqref{eq:dcedrh}, while $\sj$ is the inclusion of $C^\infty(B_pG)$ into the space of 
 	maps $\beta\colon B_pG\to \Omega^0_\F(G/K)$ such that $\beta(g_1,\ldots,g_p)\in C^\infty(\bar\tz^{-1}(m))$ is constant on $\bar\tz^{-1}(m)$, for any given $(g_1,\ldots,g_p)$ with $m=\sz(g_p)$.\\

Suppose now that $(G,K)$ admits a tubular structure. We transport the scalar multiplication in $A/B$ to a retraction along $\bar\tz$-fibers 
 \begin{equation}\label{eq:lambdaretraction}
  \lambda\colon [0,1]\times  G/K\to  G/K,\ (t,gK)\mapsto \lambda_t(gK).
  \end{equation}
Here, $\lambda_0=\bar\uz\circ \bar\tz$.
  The retraction determines a homotopy operator 
\[ T\colon \Omega^q_\F(G/K)\to \Omega^{q-1}_\F(G/K)\]
given by pullback under $\lambda$ followed by integration over $[0,1]$. 
This has the properties $T\circ T=0$ and 
\[ T\beta|_K=0.\]
It defines a vertical homotopy operator $\K$ on the double complex, where
\[ (\K\beta)(g_1,\ldots,g_p)=(-1)^p T(\beta(g_1,\ldots,g_p)).\]
That is, $[\K,\d]=1-\sj\circ \sq$, where for $\beta$ of bidegree $(p,0)$
\[ (\sq\beta)(g_1,\ldots,g_p)=\beta(g_1,\ldots,g_p)(\sz(g_p)K).\] 
The properties of $T$ imply that 
\[ \K\circ \K=0,\ \ \sq\circ \K=0.\]

By the Perturbation Lemma \ref{lem:perturbed}, we obtain a cochain map 
\begin{equation}\label{eq:vanEst} R_{G/K}=\sq\circ (1+\delta\K)^{-1}\circ \si\colon \sC^\bullet(A)_{K-\mathrm{basic}}\to \sC^\bullet(G). \end{equation}
This formula together with Proposition \ref{thm:vanest} imply that the relative cohomology of the Lie algebroid computes the differentiable cohomology of the Lie groupoid, as stated precisely below: 

\begin{theorem}\label{prop:rightinversegk} Let $K$ be a proper Lie subgroupoid of $G$ and choose a tubular structure for $(G,K)$.
The van Est integration map $R_{G/K}$ is a homotopy equivalence, which is a right inverse to the van Est differentiation map $VE_{G/K}$
\[VE_{G/K}\circ R_{G/K}=\mathrm{id}_{\sC(A)_{K-\mathrm{basic}}}.\]
 Moreover, if the tubular structure comes from a Cartan decomposition as in Remark \ref{rmk:Cartan}, then $R_{G/K}$ takes values in $\widetilde{\sC}^\bullet(G)^K$,
is a right inverse to the restriction $VE_G\colon \widetilde{\sC}^\bullet(G)^K\to \sC^\bullet(A)_{K-\mathrm{basic}}$, and induces an algebra isomorphism in cohomology.
\end{theorem}
%
%

\begin{remark}\label{rmk:integration_groups} Following Example \ref{exm:semisimple}, let $G$ be a Lie group with finitely many connected components and Lie algebra $\g$, and let $K\subset G$ be a maximal compact subgroup. In this context, the van Est integration map \eqref{eq:vanEst}
$$R_{G/K}\colon \sC^\bullet(\g)_{K-\mathrm{basic}}\to \sC^\bullet(G)$$ and the first part of Theorem \ref{prop:rightinversegk} first appeared in \cite[Theorem 3.6]{Eckhard_2} (see also \cite{locales}), where also an explicit formula was given in terms of the tubular structure of $(G,K)$. As a bonus from Theorem \ref{prop:rightinversegk} we obtain that, if we are in the setting of Example \ref{exm:semisimple} (e.g.\ if $G$ reductive as well), then $R_{G/K}$ takes values in the normalized $K$-invariant subcomplex $ \widetilde{\sC}^\bullet(G)^K$ and it becomes an algebra map when passing to cohomology. 

To recall the formula for $R_{G/K}$, we denote by $\lambda_t=\lambda(t,\cdot)\colon G/K\to G/K$ the induced scalar multiplication by $t$, with $\lambda$ as in \eqref{eq:lambdaretraction}. For $(g_1,\ldots,g_p)\in G^p$ and $(t_1,\ldots,t_p)\in [0,1]^p$, let 
\begin{equation}\label{eq:gammaformula0}
\gamma^{(p)}_{t_1,\ldots,t_p}(g_1,\ldots,g_p)=\big(\lambda_{t_1}\circ L(g_1)\cdots \circ \lambda_{t_p}\circ L(g_p)\big)(eK).
\end{equation}
For a fixed $(g_1,\ldots,g_p)\in G^p$, this defines a map $\gamma^{(p)}(g_1,\ldots,g_p)\colon [0,1]^p\to G/K$. On the other hand,
\begin{equation}\label{eq:form}\sC^\bullet(\g)_{K-\mathrm{basic}}\simeq \Omega^\bullet(G/K)^L,\ \ \alpha\mapsto \alpha_{G/K}:=\alpha\circ\theta,\end{equation}
where $\alpha_{G/K}\in\Omega^\bullet(G/K)^L$ is the $G$-invariant form obtained by precomposing $\alpha$ with the left invariant Maurer-Cartan form $\theta$ of $G$ and then passing to the quotient $G/K$. 
The van Est integration map is then given by the formula 
\begin{equation}\label{eq:veintk}
R_{G/K}(\alpha)(g_1\ldots,g_p)=\int_{[0,1]^p}\gamma^{(p)}(g_1,\ldots,g_p)^*\alpha_{G/K}.\end{equation} 
\end{remark}

\begin{proof}[Proof of Theorem \ref{prop:rightinversegk}]
By Proposition \ref{thm:vanest}, $VE_{G/K}$ is obtained from the Perturbation Lemma using the double complex  $\sD^{\bullet,\bullet}(G)_{K-\on{basic}}$ with horizontal homotopy $\mathsf{h}_K$ and projection\linebreak $\sp_K\colon \sD^{0,\bullet}(G)_{K-\on{basic}}\to \sC^\bullet(A)_{K-\on{basic}}$. Then by Lemma \ref{lem:backandforth} for the first part it suffices to show that 
$\mathsf{h}_K\circ \K=0$ and $\sp_K\circ \K=0$. Given $\beta$ as in \eqref{eq:betamap}, 
\[ (\mathsf{h}_K\beta)(g_1,\ldots,g_{p-1})(gK)=(-1)^p \int_{\tz^{-1}(m)}\ L((ga)^{-1})^*\beta(g_1,\ldots,g_{p-1},ga)\ \mu_{m}(a);\]
hence $\mathsf{h}\beta=0$ when $\beta(g_1,\ldots,g_p)|_{\sz(g_p)K}=0$ for all $(g_1,\ldots,g_p)\in B_pG$. In particular, this applies when $\beta$ is in the range of $T$. This shows $\mathsf{h}_K\circ \K=0$; the argument for $\sp_K\circ \K=0$ is similar.

To see that $R_{G/K}$ takes values in $\sC^\bullet(G)^K$ when having a Cartan decomposition, denote by $\sD^{p,q}(G,K)_{K\mathrm{-basic}}\subset \sD(G)^{p,q}_{K\mathrm{-basic}}$ the space of invariant elements $\beta$ w.r.t.\ the $p+1$ $K$-actions \eqref{eq:actions2} (i.e.\ we restrict the $G$-actions to $K$), and so that $\beta(g_1,\ldots,g_p)=0$ whenever $g_i\in K$ for some $i$. Note that $\sq$ maps $\sD^{p,0}(G,K)_{K\mathrm{-basic}}$ to $\widetilde{\sC}^p(G)^K$. We claim that $\sD^{\bullet,\bullet}(G,K)_{K\mathrm{-basic}}$ is a subcomplex of $\big(\sD(G)^{\bullet,\bullet}_{K\mathrm{-basic}},\delta, \d_{CE}\big)$, and that 
the rest of the maps $\K,\si$ involved in the definition of $R_{G/K}$ send $\sD(G,K)_{K\mathrm{-basic}}$ to $\sD(G,K)_{K\mathrm{-basic}}$. Indeed, for the last of the $p+1$ action we have that 
$$(a\cdot \beta)(g_1,\ldots,g_p)=L(a^{-1})^*\beta(g_1,\ldots g_pa)$$
for $a\in K$, $\tz(a)=\sz(g_p)$. A direct computation then shows that $\sD(G,K)_{K\mathrm{-basic}}$ is closed under $\delta$ and $\d_{CE}$ as in \eqref{eq:deltabeta} and \eqref{eq:dbeta}, respectively;  similarly, $\si$ maps $\sC^\bullet(A)_{K-\mathrm{basic}}=\Omega^q_\mathcal{F}(G/K)^{L}$ to $\sD^{0,q}(G,K)_{K\mathrm{-basic}}$ -- the space of left $K$-invariant foliated forms of $G/K$ --. Moreover, the tubular structure \eqref{eq:tubular} of $G/K$ is $K$-equivariant; hence, 
\[\lambda_t(a\cdot g K)=a\cdot\lambda_t(gK)\]
for $a\in K$, $\sz(a)=\tz(g)$. This implies that $$L(a^{-1})^*T(\beta(g_1,\ldots,g_pa))=T(L(a^{-1})^*\beta(g_1,\ldots,g_pa)).$$
With this, a direct computation shows that $\sD(G,K)_{K\mathrm{-basic}}$ is closed under $\K$, hence implying that  $R_{G/K}$ lands in $\widetilde{\sC}^\bullet(G)^K$. To see that in this case $R_{G/K}$ is an algebra map in cohomology, note that in this case  
$VE_G|_{\widetilde{\sC}^\bullet(G)^K}\circ R_{G/K}=\mathrm{id}.$
As $VE_G|_{\widetilde{\sC}^\bullet(G)^K}$ is an algebra map by Theorem \ref{cor:VE_relative},
we then obtain that, when passing to cohomology, $R_{G/K}$ is an isomorphism with inverse the algebra map $VE_G|_{\widetilde{\sC}^\bullet(G)^K}$. This implies that in cohomology $R_{G/K}$ is an algebra map.
%
\end{proof}

  \subsection{The proper and transitive case}\label{section:properandtransitive} 
  
By Example \ref{exm:transitive} we know that any transitive Lie groupoid $G$ satisfying a very mild assumption admits a proper transitive subgroupoid $K$ such that the pair $(G,K)$ has a tubular structure as in Definition \ref{def:tubular} (e.g.\  when the isotropy Lie groups of $G$ have finitely many connected components). Moreover, in some cases the tubular structure comes from a Cartan decomposition as in Definition \ref{def:Cartan} (e.g. when the isotropy Lie groups of $G$ are non-compact semisimple). In this section we put together some results from the proper and transitive cases, and follow the notation of Section \ref{transitive}.\\

Assume that $K$ is a proper and transitive Lie subgroupoid of $G$ and fix $z\in M$. Recall from \cite[Appendix 10.4]{Joao} that any measure $\mu_M$ on $M$ determines a Haar system $\mu$ on $K$ by the formula
\[\int_{\tz^{-1}(m)}f(a)\mu_m(a)=\int_M\left(\int_{K_m}f(ka)  \d k\right)\mu_M;\]
  here $\d k$ is the normalized invariant Haar measure on the compact Lie group $K_m$ and $a\in\tz^{-1}(m) \mapsto \int_{K_m}f(ka)  \d k$ is regarded as a function on $M=\tz^{-1}(m)/K_m$ as it is $K_m$-invariant. In particular we may take $\mu_M$ to be the delta-distribution at $z\in M$. The resulting properly supported normalized Haar system becomes
  \begin{equation}\label{eq:measure}\int_{\tz^{-1}(m)}f(a)\mu_m(a)=\int_{K_m}f(kh)  \d k\end{equation}
where, by the invariance of $\d k$, $h\in K$ is {\it any} arrow so that $\sz(h)=z$ and $\tz(h)=m.$ \\

For the remaining results in this section, we fix $z\in M$ and set the following notation: The isotropy Lie group of $G$ (respectively $K$) at $z$ is denoted by $G_z$ (respectively $K_z$). Analogously, $A$ denotes the Lie algebroid of $G$ and the isotropy Lie algebra of $A$ at $z$ is denoted by $\g_z$.
  
  \begin{lemma}\label{lemma:Av} Suppose $K$ is a proper and transitive subgroupoid of $G$. The following diagram commutes
  \begin{equation*}\label{eq:diagram1} \xymatrix{ 
\sC^p(G)  \ar[r]^-{\mathrm{Av} } \ar[d]_r&   \sC^p(G)^{K} \ar[d]^r  \\
\sC^p(G_z) \ar[r]^-{\mathrm{Av}}& \sC^p(G_z)^{K_z}.  
 }
 \end{equation*}
Here, $r$ denotes the restriction map, the bottom averaging operator is associated to the normalized invariant Haar measure $\d k$ on $K_z$, and the upper averaging operator to the Haar system defined by \eqref{eq:measure}.
  \end{lemma}
  
  \begin{proof} For $f\in C^\infty(B_pG)$ and $(g_1,\ldots,g_p)\in B_pG_z$, we have that $r\circ \mathrm{Av^{(i)}}(f)(g_1,\ldots g_p)= \mathrm{Av^{(i)}}(f)(g_1,\ldots g_p),$ where $\mathrm{Av^{(i)}}$ is the averaging operator w.r.t. the $i$-th $K$-action induced by the Haar system \eqref{eq:measure}. For $0<i<p,$
\[\begin{split}
\mathrm{Av^{(i)}}(f)(g_1,\ldots g_p)&=\int _{\tz^{-1}(z)\cap K}f(g_1,\ldots,g_i a,a^{-1} g_{i+1},\ldots,g_p)\mu_z(a)\\
&= \int _{K_z}f(g_1,\ldots,g_i (kh),(kh)^{-1} g_{i+1},\ldots,g_p)\d k
\end{split}\]
with $h\in K$ any arrow with $\sz(h)=\tz(h)=z$ (see \eqref{eq:measure}). Letting $h=z$, the above equation becomes 
\[\begin{split}
\int _{K_z}f(g_1,\ldots,g_i k,k^{-1} g_{i+1},\ldots,g_p)\d k= \mathrm{Av^{(i)}}(r(f))(g_1,\ldots g_p),
\end{split}\]
where now $\mathrm{Av^{(i)}}$ is the one induced by $\d k$. For $i=0$ and $p$, we again obtain that $r\circ \mathrm{Av^{(i)}}(f)=\mathrm{Av^{(i)}}\circ r(f)$. The conclusion now follows by consecutively applying the $i$-th averaging operators so as to obtain that $r\circ \mathrm{Av}(f)=\mathrm{Av}\circ r(f)$ (see \eqref{eq:averaging}).
  \end{proof}  
  
When dealing with pairs $(G,K)$ that admit a tubular structure $\phi$ (or a Cartan decomposition), the restriction $\phi_z$ to their isotropies at $z$ is also a tubular structure (respectively a Cartan decomposition) of $(G_z,K_z)$. These are the tubular structures considered in the following result:

\begin{lemma}\label{lemma:R}
Suppose that $K$ is a proper and transitive subgroupoid of $G$ such that $(G,K)$ has a Cartan decomposition. Then the following diagrams commute
\[ \xymatrix{ 
\sC^p(A)_{K-\mathrm{basic}}  \ar[r]^-{R_{G/K} } \ar[d]_r&   \widetilde{\sC}^p(G)^{K} \ar[d]^r  \\
\sC^p(\g_z)_{K_z-\mathrm{basic}}\ar[r]^-{R_{G_z/K_z}}& \widetilde{\sC}^p(G_z)^{K_z}  
 }
 \xymatrix{ 
& \sC^p(A)_{K-\mathrm{basic}}  \ar[r]^-{R_{G/K} } &   \widetilde{\sC}^p(G)^{K} \\
& \sC^p(\g_z)_{K_z-\mathrm{basic}}\ar[r]^-{R_{G_z/K_z}} \ar[u]^{\mathrm{ext}}& \widetilde{\sC}^p(G_z)^{K_z}  \ar[u]_{\mathrm{ext}}.
 }
\]
\end{lemma}

\begin{proof} 
The left diagram clearly commutes. The right diagram commutes because $r$ and $\mathrm{ext}$ are inverse to each other, and the commutativity of the left diagram.
\end{proof}

Passing to the van Est differentiation maps, we have the following general result for transitive Lie subgroupoids, whose proof is analogous to that of Lemma \ref{lemma:R}.

\begin{corollary}\label{cor:r-ve}
Let $K$ be a transitive Lie subgroupoid of $G$. Then, the diagrams
\[
\xymatrix{ 
\sC(A)^p_{K-\mathrm{basic}} \ar[d]_{r}&   \sC^p(G)^{K} \ar[d]^{r}  \ar[l]_-{VE_{G}} \\
\sC^p(\g_z)_{K_z\mathrm{-basic}} & \sC^p(G_z)^{K_z} \ar[l]_-{VE_{G_z}}
 }
\xymatrix{ 
&\sC(A)^p_{K-\mathrm{basic}}  &   \sC^p(G)^{K}  \ar[l]_-{VE_G }\\
&\sC^p(\g_z)_{K_z\mathrm{-basic}} \ar[u]^{{\mathrm{ext}}} & \sC^p(G_z)^{K_z} \ar[l]_-{VE_{G_z}} \ar[u]_{{\mathrm{ext}}} 
 }
\]
 commute (an analogous result holds for the normalized invariant subcomplexes). Here, the right and left $\mathrm{ext}$ are given by equations \eqref{eq:ext} and \eqref{eq:ext_algebroid}, respectively.
\end{corollary}

%

  For the following result, we consider the Haar system \eqref{eq:measure} on $K$ and the invariant Haar measure $\d k$ on $K_x$.  We also consider the tubular structure of $(G,K)$ given by the Cartan decomposition and its induced tubular structure on $(G_z,K_z).$ 

\begin{theorem}\label{thm:transitive} Let $K$ be a proper and transitive subgroupoid of $G$ such that $(G,K)$ admits a Cartan decomposition. Then the following is a commutative diagram of homotopy equivalences

 \[
 \xymatrix{ 
\sC^\bullet(G)  \ar[r]^-{VE_{G/K} } \ar[d]_r&   \sC^\bullet(A)_{K-\mathrm{basic}}\ar[d]^r  \\
\sC^\bullet(G_z) \ar[r]^-{VE_{G_z/K_z}}& \sC^\bullet(\g_z)_{K_z-\mathrm{basic}}.
 }
\]
In particular, \[r\circ VE_{G/K}\colon \sC^\bullet(G)\to \sC^\bullet(\g_z)_{K_z-\mathrm{basic}} \ \ \mathrm{and}\ \ \mathrm{ext}\circ R_{G_z/K_x}\colon \sC^\bullet(\g_z)_{K_z-\mathrm{basic}}\to \sC^\bullet(G)\] are homotopy equivalences with $(r\circ VE_{G/K})\circ(\mathrm{ext}\circ R_{G_z/K_x})=\mathrm{id}_{\sC^\bullet(g_z)_{K_z-{\mathrm{basic}}}}$.

\end{theorem}

\begin{proof}First, one can check that for transitive groupoids the inclusion $\iota\colon G_z\to G$ is an essential equivalence (i.e.\ a Lie groupoid map which induces a Morita equivalence, see \cite{Crainic:vanEst} for the definitions). Hence, by \cite[Theorem 1]{Crainic:vanEst} the map $r\colon \sC^\bullet(G)\to \sC^\bullet(G_z)$ induced by $\iota$ at cohomology, is a homotopy equivalence. That the rest of the maps involved in the diagram and $R_{G_z/K_x}$ are homotopy equivalences is a consequence of Theorems \ref{thm:ext}, \ref{thm:ext_2} and \ref{prop:rightinversegk}. The commutativity of the diagram is Lemma \ref{lemma:Av} together with Corollary \ref{cor:r-ve}.

Also, by the commutativity of the diagram above, and Theorems \ref{thm:ext} and \ref{prop:rightinversegk},
\begin{align*}(r\circ VE_{G/K})\circ(\mathrm{ext}\circ R_{G_z/K_x})&=(VE_{G_z/K_z}\circ r)\circ(\mathrm{ext}\circ R_{G_z/K_x})\\
&=VE_{G_z/K_z}\circ R_{G_z/K_x}=\mathrm{id}_{\sC^\bullet(g_z)_{K_z-\mathrm{basic}}}.\end{align*}

To conclude the proof, note that $\mathrm{ext}\circ R_{G_z/K_x}\colon \sC^\bullet(g_z)_{K_z-\mathrm{basic}}\to \sC^\bullet(G)$ can be written as the composition
\[\sC^\bullet(g_z)_{K_z-\mathrm{basic}}\overset{R_{G_z/K_x}}{\longrightarrow} \widetilde{\sC}^\bullet(G_z)^{K_z} \overset{\mathrm{ext}}{\longrightarrow}  \widetilde{\sC}^\bullet(G)^{K}\overset{\iota}{\longrightarrow} {\sC}^\bullet(G),\]
where all the maps are homotopy equivalences ($\iota$ by Theorem \ref{thm:normalized}, and $\mathrm{ext}$ by Theorem \ref{thm:ext}). 
\end{proof}

\subsection{Example: The polar decomposition and the cohomology of $GL_n(\C)$}\label{section:example}

Let $GL_n=GL_n(\C)$ and let $\gln$ be its Lie algebra. We will explicitly find generators of $H(GL_n)$ at the cochain level, described in terms of cochains in $\sC^\bullet(\gln)_{U(n)-\mathrm{basic}},$ where $U(n)=\{B\in GL_n\mid BB^*=I\}$. We will also describe the van Est map $VE_{GL_n}\colon\sC^\bullet(GL_n)\to\sC^\bullet(\gln)$ in terms of the cochain generators of the respective cohomologies. 

In the first part, we find explicit generators for $H(\gln,U(n))$ at the cochain level. Then, setting $G=GL_n$ and $K=U(n)$, one can apply Theorem \ref{prop:rightinversegk} to find the generators of $H(GL_n)$ as the image of the generators of $H(\gln,U(n))$ via the Est integration map \linebreak
$R_{GL_n/U(n)}\colon  \sC^\bullet (\gln)_{U(n)-\mathrm{basic}}\to \sC^\bullet(GL_n).$
\\

First of all, recall from \cite{kamber} that 
\[H(\gln)= \wedge(u_1',u_3',\ldots, u'_{2n-1}, b_1,b_3,\ldots, b_{2n-1}),\]
where the $u's$ and $b's$ are given (at the cochain level) by
 $u_q=\mathrm{Re}(\Phi_{2q-1}),\ b_q=\mathrm{Im}(\Phi_{2q-1})$ for $q$ odd and $u'_q=\mathrm{Im}(\Phi_{2q-1}),\ b_q=\mathrm{Re}(\Phi_{2q-1})$ for $q$ even, with
\begin{equation*}\Phi_{2q-1}(A_1,\ldots,A_{2q-1})=\sum_{\gamma\in S_{2p-1}}\mathrm{sgn}(\gamma)\mathrm{tr} (A_{\gamma(1)}\cdots A_{\gamma(2q-1)})
\end{equation*}
being the $\Ad$--invariant skew-symmetric multilinear maps on $\gln$. In turn, 
\[H(\gln,U(n))=(\wedge^\bullet\mathfrak{p}^*)_{U(n)-\mathrm{inv}}=\wedge(u_1,u_3,\ldots,u_{2n-1}),\]
where the $u's$ are given (at the cochain level) by
\begin{equation}\label{eq:u}u_{2q-1}=i^{q-1}\Phi_{2q-1}\in C^{2q-1}(\gln)_{U(n)-\mathrm{basic}}.\end{equation}

Note that the inclusion map 
\begin{equation}\label{eq:inc_H}\iota\colon H(\gln,U(n))\to H(\gln)\end{equation}
sends $u_{2q-1}$ to $u'_{2q-1}$ and it is injective. \\

For the second step, we use the {\it polar decomposition} of $GL_n$ (see Example \ref{exm:semisimple}). It comes from the polar decomposition of $\gln$ as  
\[\gln=\mathfrak{u}(n)\oplus \mathfrak{p},\ \ \ A=\frac{1}{2}(A-A^*)+\frac{1}{2}(A+A^*),\ A\in\gln.\]
with $\mathfrak{u}(n)=\{A\in\gln\mid A^*=-A\}$ -- the Lie algebra of $U(n)$ --, while $\mathfrak{p}=\{A\in\gln\mid A^* = A\}$. Note that $\mathfrak{u}(n)$ and $\mathfrak{p}$ correspond to the eigenspace  $1$ and $-1$, respectively, of the Cartan involution $\theta(X)=-X^*$. The polar decomposition of $GL_n$ is then defined as the $U(n)$-equivariant diffeomorphism
\[\mathfrak{p}\times U(n)\to GL_n,  \ \ (X,U)\mapsto e^XU,\]
 thus it is a Cartan decomposition for the pair $(GL_n, U(n)).$ Applying Theorem \ref{prop:rightinversegk} we obtain the following description of $H(GL_n)$.

\begin{theorem}\label{eq:group} 
The van Est integration map \eqref{eq:vanEst}
\[R_{GL_n/U(n)}\colon  \sC^\bullet (\gln)_{U(n)-\mathrm{basic}}\to \sC^\bullet(GL_n)\]
is a homotopy equivalence and it induces an algebra isomorphism in cohomology. In particular, 
\begin{equation*}H(GL_n)=\wedge(v_1,\ldots, v_{2n-1}) \end{equation*}
where $v_{2p-1}:=R_{GL_n/U(n)}(u_{2p-1})$ has degree $2p-1$ and is a normalized $U(n)$-invariant cochain (in the sense of Remark \ref{rmk:normalized}).
\end{theorem}

With the previous result at hand, we obtain the following description of the degree one generator $v_1\in \sC^1(GL_n)$:

\begin{corollary}\label{cor:group}
The first cohomology class $v_1=R_{GL_n/U(n)}(u_1)\in H^1(GL_n)$ is explicitly given at the cochain level by
\[v_1(A)=\mathrm{tr}(X)\in\R,\ \ \ A\in GL_n,\]
where $A=e^XU\ (X\in\mathfrak{p})$ is the polar decomposition of the invertible matrix $A$.
\end{corollary}

\begin{proof}

First, we identify the quotient $GL_n/U(n)$ with the space of positive-definite Hermitian matrices $P=\exp_{GL_n}(\mathfrak{p})$. The scalar multiplication of $P$ is inherited by the diffeomorphism given by the restriction of the exponential $e\colon \mathfrak{p}\to P$:
\begin{equation}\label{eq:lambda}\lambda_t\colon P\to P,\ \ e^X\mapsto e^{tX}\end{equation}
We apply $R_{GL_n/U(n)}$ as in \eqref{eq:veintk} to $u_1=\mathrm{tr}|_\mathfrak{p}$. 
For this, notice that under the identification $C^1(\gln,\mathfrak{u}(n))\simeq\Omega^1(P)^{GL_n}$ given in \eqref{eq:form}, the $1$-form $\mathrm{tr}_{P}=\mathrm{tr}_{GL_n/U(n)}$ is written as the composition 
\[\mathrm{tr}_{P}\colon  TP\overset{\theta}{\longrightarrow}\gln=\mathfrak{p}\oplus\mathfrak{u}(n)\overset{\mathrm{tr}|_{\mathfrak{p}}}{\longrightarrow}\R\] 
(all diagonal entries of a Hermitian matrix are real), where $\theta$ is the left invariant Maurer Cartan form restricted to $TP$. Now, the map $\gamma_A\colon [0,1]\to P, \ A\in GL_n$ appearing in $R_{GL_n/U(n)}$ is just 
\[\gamma_A(t)=\lambda_t(AU(n))=\lambda_t(e^XUU(n))=\lambda_t(e^XU(n))=e^{tX}\]
where $A=e^XU$ is the polar decomposition of $A$. We have that
\begin{equation*}
\begin{split}
\gamma_A^*(\mathrm{tr}_P)\left(\frac{d}{dt}\right)=\mathrm{tr}_P\circ \gamma_{A,*}\left(\frac{d}{dt}\right)=\mathrm{tr}_P\left(\frac{d}{dt}e^{tX}\right)=\mathrm{tr}_P(L_{e^{tX}}(X))=\mathrm{tr}(X),
\end{split}
\end{equation*}
where the third equality follows by using that $t\mapsto e^tX\in GL_n$ is an integral curve of the left invariant vector field determined by $X\in\mathfrak{p}\subset\gln$, and the last equality by left-invariance of $\mathrm{tr}_P$. Thus
\[R_{GL_n/U(n)}(u_1)(A)=\int_{[0,1]}\gamma_A^*(\mathrm{tr}_P)\left(\frac{d}{dt}\right)=\int_{[0,1]}\mathrm{tr}(X)=\mathrm{tr}(X).\] \end{proof}

As a byproduct of the discussion above, we have the following result:
\begin{corollary}\label{cor:injective}
The van Est map $VE_{GL_n}: H(GL_n)\to H(\gln)$ is injective.
\end{corollary}

\begin{proof}
This is a direct consequence of the commutativity of the diagram 
\[
\xymatrix{ 
\widetilde{\sC}^\bullet(GL_n)^{U(n)} \ar[d]_{VE_{GL_n}} \ar[r]^-\iota &   \sC^\bullet(GL_n)\ar[d]^{VE_{GL_n}}  \\
\sC^\bullet(\gln)_{U(n)\mathrm{-basic}} \ar[r]^-\iota & \sC^\bullet(\gln),
}
 \]
 and the facts that the lower map is injective in cohomology (see \eqref{eq:inc_H}), and that the upper and left vertical maps are homotopy equivalences by 
 Theorems \ref{thm:averaging} and \ref{prop:rightinversegk}, respectively. 
\end{proof}

\section{Characteristic classes}\label{section:cc}

In this section, we introduce the definition of characteristic classes of a representation of a Lie groupoid, and relate them with the existing notion of characteristic classes of a representation of a Lie algebroid (see \cite{Crainic:vanEst,CrainicFernandes:classes,Fernandes:classes}).\\ 

Throughout, we deal with cohomologies with real coefficients.
  
\subsection{Background on representations}
A \emph{representation} of a Lie groupoid $G\rra M$ is a (left) $G$-action on a (real or complex) vector bundle $V\to M$ so that for any $g\in G$ with $\sz(g)=m,\tz(g)={m'}$, the induced map $V_m\to V_{m'}, \ e\mapsto ge$ is a linear isomorphism.  Alternatively, a representation can be encoded as follows: 
Associated to any vector bundle $V\to M$, we consider the gauge groupoid $GL(V)\rra M$ of linear isomorphism between the fibers
\[GL(V)=\{\phi\colon V_m\to V_{m'}\mid m,m'\in M\ \mathrm{and}\ \phi\ \mathrm{is\ a\ linear\ isomorphism}\},\]
with source and target $\sz(\phi)=m,\tz(\phi)=m'$, and multiplication given by composition. 
A representation of a Lie groupoid $G\rra M$ is equivalently described as a Lie groupoid morphism 
 \[\Phi\colon G\to GL(V)\]
 over the identity map of $M$.
 
 Passing to the infinitesimal picture, a representation of a Lie algebroid $A\to M$ on a vector bundle $V\to M$ is a connection-like operator $\nabla:\Gamma(A)\times\Gamma(V)\to \Gamma(V)$, satisfying the Leibniz rule w.r.t. the anchor map, and so that it is flat. In analogy with the groupoid case a representation can be alternatively described as a Lie algebroid morphism
 \[\Psi\colon A\to \mathfrak{gl}(V)\]
over the identity map, where $\mathfrak{gl}(V)$ is the Lie algebroid of $GL(V)$. 

If $A$ is the Lie algebroid of $G$, then any representation $\Phi:G\to GL(V)$ induces a representation of $A$ by
\begin{equation}\label{eq:induced}\Psi:=\mathrm{Lie}(\Phi)\colon A\to \mathfrak{gl}(V)\end{equation}
with $\mathrm{Lie}(\Phi)=d\Phi|_A.$\\

The following result is used and proved in Section \ref{section:characteristic}, and may be interesting in its own right. We remind the reader that we are dealing with real coefficient in the cohomology.

\begin{corollary}\label{cor:VE-injective} For a complex vector bundle $V,$ the van Est map $VE_{GL(V)}\colon H(GL(V))\to H(\mathfrak{gl}(V))$ is injective.
\end{corollary}

\subsection{Characteristic classes for algebroids revisited}\label{section:characteristic-algebroid}
In this section we focus on an alternative, simple and direct definition of the characteristic classes of representations of Lie algebroids first defined in \cite{Crainic:vanEst} (see also \cite{CrainicFernandes:classes,Fernandes:classes}). Our approach relies on the notion of relative cohomology and its properties. \\

Let $A$ be a Lie algebroid over $M$, $V$ a complex vector bundle over $M$ of rank $n$, and 
 \[\Psi\colon A\to \mathfrak{gl}(V)\]
a representation of $A$ on $V$. Fix a (fiberwise) Hermitian metric $h$ on $V$ and denote by $U(h)\subset GL(V)$ the transitive and proper Lie subgroupoid  
\begin{equation*}
U(h)=\{\phi\colon V_m\to V_{m'}\in GL(V)\mid \phi\colon (V_m,h_m)\to (V_{m'},h_{m'})\  \mathrm{is\ an\ isometry}\}.
\end{equation*}
Fix a point $z\in M$ and an isometry
  \begin{equation}\label{eq:F}F\colon (\C^n, h_0)\to (V_z,h_z),\end{equation}
where $h_0$ is the standard Hermitian metric of $\C^n$. Note that conjugation by $F$
\[C_F\colon GL(V_z)\to GL_n,\ \ C_F(\phi):= F^{-1}\circ\phi\circ F\]
is a Lie group isomorphism sending the Lie subgroup $U(h_z)$ onto the unitary group $U(n)$.

Consider the following composition of cochain maps
\begin{equation}\label{eq:ch}
\xymatrix{
M_\Psi\colon\sC^\bullet(\gln)_{U(n)-\mathrm{basic}} \ar[r]^-{C_F^*} & \sC^\bullet(\mathfrak{gl}(V_z))_{U(h_z)-\mathrm{basic}} \ar[r]^-{\mathrm{ext}}& \sC^\bullet(\mathfrak{gl}(V))_{U(h)-\mathrm{basic}}\ar[r]^-{\Psi^*} & \sC^\bullet(A).
}
\end{equation}
Recall from Section \ref{section:example} that
$$H(\gln,U(n))=\wedge(u_1,u_3,\ldots,u_{2n-1}),$$
where the $u's$ are explicitly defined at the cochain level by trace operators \eqref{eq:u}.

\begin{definition}\label{def:classes} The {\it characteristic classes} of the representation of the Lie algebroid $A$ on the complex vector bundle $V$  are defined as
 \[u_{2p-1}(V):=[M_\Psi(u_{2p-1})]\in H^{2p-1}(A), \ \ 1\leq p\leq n, \]
 where $n$ is the rank of $V$. If $V$ is a real vector bundle then they are defined as $u_{2p-1}(V):=u_{2p-1}(V_\C),$
 where $V_\C$ is the complexification of $V$. 
\end{definition}

\begin{remark}
This definition should be compared with the explicit approach in \cite{Crainic:vanEst}. By inspection, Definition \ref{def:classes} agrees with \cite[formula (45)]{Crainic:vanEst} in local coordinates.
\end{remark}

\begin{lemma}\label{prop:ch}
The characteristic classes of the representation $\Psi$ do not depend on the choice of $z\in M$, $F$ as in \eqref{eq:F}, or the Hermitian metric $h$ on $V$. 
\end{lemma}

To prove Lemma \ref{prop:ch} we will need the following result: 

\begin{lemma}\label{lemma:ch} Fix $z\in M$ and $F$ as in \eqref{eq:F}, then 
\begin{itemize}[leftmargin=*]
\item for any other point $m\in M$ and isometry $F'\colon (\C^n, h_0)\to (V_m,h_m)$, $$\mathrm{ext}_z\circ C_F^*=\mathrm{ext}_m\circ C_{F'}^*;$$
\item for any other Hermitian metric $h'$ of $V$, and any fiberwise isometry $H:(V,h)\to (V,h')$ covering the identity,
$$C_H^*\colon\sC(\mathfrak{gl}(V))\to \sC(\mathfrak{gl}(V))$$
is homotopic to the identity. 
\end{itemize}
\end{lemma}

\begin{proof}[Proof of Lemma \ref{prop:ch}]
By the first item of Lemma \ref{lemma:ch} we already know that the characteristic classes do not depend on the choice of either point or isometry \eqref{eq:F}, so fix $z\in M$ and the isometry $F$ as in \eqref{eq:F}. Denote the inclusions by
\[\iota_h\colon\sC(\mathfrak{gl}(V))_{U(h)-\mathrm{basic}}\to \sC(\mathfrak{gl}(V)),  \ \ \iota_{h'}\colon\sC(\mathfrak{gl}(V))_{U(h')-\mathrm{basic}}\to \sC(\mathfrak{gl}(V)).\] 
Consider the diagram 
\[
\xymatrix{
& \sC(\mathfrak{gl}(V_z))_{U(h_z)-\mathrm{basic}} \ar[r]^-{\mathrm{ext}} \ar[dd]^-{C_{H_z}^*} & \sC(\mathfrak{gl}(V))_{U(h)-\mathrm{basic}} \ar[r]^-{\iota_h} \ar[dd]^-{C_{H}^*}& \sC(\mathfrak{gl}(V))  \ar[dd]^-{C_H^*}\\
\sC(\gln)_{U(n)-\mathrm{basic}} \ar[ru]^-{C_F^*} \ar[rd]^-{C_{F'}^*} & & & \\
& \sC(\mathfrak{gl}(V_z))_{U(h'_z)-\mathrm{basic}} \ar[r]^-{\mathrm{ext}} & \sC(\mathfrak{gl}(V))_{U(h')-\mathrm{basic}} \ar[r]^-{\iota_{h'}} & \sC(\mathfrak{gl}(V)),
}
\]
with $F':=H_z\circ F:(\C^n,h_0)\to (V_z,h'_z)$. Then the left triangle is clearly commutative, and also the right square is commutative. The middle square is also commutative as $C_{H_z}^*\circ r=r\circ C_H^*$, where $r$ are the restriction maps, and the fact that $r$ and $\mathrm{ext}$ are inverse to each other. 

Denote by $u_{2p-1}(V,h)$ ($u_{2p-1}(V,h')$ respectively) the characteristic classes of Definition \ref{def:classes} defined by the composition \eqref{eq:ch} using the Hermitian metric $h$ ($h'$ respectively). 
We compute:
\begin{equation*}
\begin{split}
u_{2p-1}(V,h)&=\Psi^*[\iota _h\circ \mathrm{ext}\circ C_F^*(u_{2p-1})]=\Psi^*[C_H^*\circ \iota _h\circ \mathrm{ext}\circ C_F^*(u_{2p-1})]\\
&=\Psi^*[\iota _{h'}\circ \mathrm{ext}\circ C_{F'}^*(u_{2p-1})]=[u_{2p-1}(V,h')],
\end{split}
\end{equation*}
where in the second equality we used the second item of Lemma \ref{lemma:ch}.
\end{proof}

\begin{proof}[Proof of Lemma \ref{lemma:ch}]
For the first item, note that, since the restriction map $$r_m\colon\sC(\mathfrak{gl}(V))_{U(h)-\mathrm{basic}} \to \sC(\mathfrak{gl}(V_m))_{U(h_m)-\mathrm{basic}}$$ is inverse to the extension map $\mathrm{ext}_m$, the desired equality holds if and only if 
$$C_{F'^{-1}}^*\circ r_m\circ\mathrm{ext}_z\circ C_F^*=\id_{\sC^\bullet(\gln)_{U(n)-\mathrm{basic}}}.$$
To check this last equality, note that we can write $r_m\circ \mathrm{ext}_z$ as the pullback by conjugation 
$$r_m\circ \mathrm{ext}_z=C_{b(m)}^*,$$
 where $b:M\to U(h)$ is any bisection with the property that $\sz(b(m))=z$ (i.e.\ $b(m):(V_z,h_z)\to (V_m,h_m)$ is an isometry).
 Therefore,
 \begin{equation*}
 C_{F'^{-1}}^*\circ r_m\circ\mathrm{ext}_z\circ C_F^*=C_{F'^{-1}}^*\circ C_{b(m)}^*\circ C_F^* =C^*_{F'^{-1}\circ b(m)\circ F}.
 \end{equation*}
 As $U:=F'^{-1}\circ b(m)\circ F$ is the composition of isometries 
 \[(\C^n, h_0)\overset{F}{\longrightarrow} (V_z,h_z)\overset{b(m)}{\longrightarrow} (V_m,h_m)\overset{F'^{-1}}{\longrightarrow} (\C^n,h_0), \]
 $C_{U}\colon GL_n\to GL_n$ is  conjugation by the unitary map $U\in U(n)$, $C^*_U:\gln^*\to \gln ^*$ is the coadjoint action by $U$, and hence $C^*_U\colon \sC(\gln)_{U(n)-\mathrm{basic}}\to \sC(\gln)_{U(n)-\mathrm{basic}} $ is the identity map on the $U(n)$-basic subcomplex, as we desired.\\
 
 For the second item, we will construct a homotopy $\mathsf{h}$ between the identity map and $C_H^*$ as follows. Denote by $GL(V)_M$ the space of linear automorphisms of $V$ fixing the base $M$. Let $t\mapsto H_t\in GL(V)_M$ be a smooth path connecting the identity and $H$:
 \[H_t\colon V\overset{\simeq}{\longrightarrow} V, \ \ H_0=\id_V,\ \ H_1=H. \]
 Observe that $GL(V)_M\subset GL(V)$ is the Lie groupoid given by the union of the isotropy Lie groups of $GL(V)$, and that $H_t\in\mathrm{Bis}(GL(V)_M)$.
 
Infinitesimally, the Lie algebroid $\mathrm{Lie}(GL(V)_M)=\mathfrak{gl}(V)_M$ --the space of linear transformations from $V$ into $V$ fixing the base $M$-- is a Lie subalgebroid of $\mathfrak{gl}(V)$ given by the kernel of the anchor map, and $H_t$ induces a path of sections $t\mapsto X_t\in \mathfrak{gl}(V)_M$ defined by 
\[\frac{d}{dt}H_t=X_t\circ H_t.\]
We can then have the induced Lie derivatives $\L_{X_t}=[\d_{CE},\iota_{X_t}]$ along $X_t$,
$\L_{X_t}\colon\sC^\bullet(\mathfrak{gl}(V))\to \sC^\bullet(\mathfrak{gl}(V))$,
satisfying the property
\begin{equation}\label{eq:derivative}
\frac{d}{dt}C_{H_t}^*(\omega)=\L_{X_t}\big(C_{H_t}^*(\omega)\big), \ \ \omega\in \sC^\bullet(\mathfrak{gl}(V)).
\end{equation}

Define $\mathsf{h}\colon\sC^\bullet(\mathfrak{gl}(V))\to\sC^{\bullet-1}(\mathfrak{gl}(V))$ by 
\[\mathsf{h}(\omega):=\int_0^1\iota_{X_t}\big(C_{H_t}^*(\omega)\big)\ dt. \]
We compute:
\begin{equation*}
\begin{split}
C_H^*(\omega)-\omega&=\int_0^1\frac{d}{dt}C^*_{H_t}(\omega)\ dt= \int_0^t  \L_{X_t}\big(C_{H_t}^*(\omega)\big) \ dt= \int_0^1\d_{CE}\iota_{X_t}\big(C_{H_t}^*(\omega)\big) +\iota_{X_t}\d_{CE}\big(C_{H_t}^*(\omega)\big) \ dt\\
&=\d_{CE}\mathsf{h}(\omega)+\mathsf{h}\d_{CE}(\omega),
\end{split}
\end{equation*}
where in the last equality we used that $\d_{CE}$ commutes with $C^*_{H_t}$ since $C_{H_t}\colon GL(V)\to GL(V)$ is a groupoid morphism for each $t$. 
\end{proof}

\subsection{Characteristic classes for Lie groupoids}\label{section:characteristic}

Let $G$ be a Lie groupoid over $M$, $V$ a complex vector bundle over $M$ of rank $n$, and 
 \[\Phi\colon G\to GL(V)\]
 a representation of $G$ on $V$. Fix a Hermitian metric $h$ of $V$, a point $z\in M$ and an isometry $F\colon(\C^n,h_0)\to (V_z,h_z)$. Using the same conventions of Section \ref{section:characteristic-algebroid}, consider the following composition of cochain maps
\begin{equation}\label{eq:classes_composition}
\xymatrix{
M_\Phi: \widetilde{\sC}^\bullet(GL_n)^{U(n)} \ar[r]^-{C_F^*} & \widetilde{\sC}^\bullet(GL(V_z))^{U(h_z)}   \ar[r]^-{\mathrm{ext}} & \widetilde{\sC}^\bullet(GL(V))^{U(h)} \ar[r]^-{\Phi^*} & \sC^\bullet (G).
}
\end{equation}
From Section \ref{section:example}, we know that 
\begin{equation}\label{eq:generators}H(GL_n)=\wedge (v_1,\ldots, v_{2n-1} ),\end{equation}
where $v_{2p-1}:=R_{GL_n/U(n)}(u_{2p-1})$ is the cochain of $\widetilde{\sC}^\bullet(GL_n)^{U(n)}$ given in Theorem \ref{eq:group}.

\begin{definition}\label{def:classes_groupoid}
The {\it characteristic classes} of the representation $\Phi$ of the Lie groupoid $G$ on the complex vector bundle $V$ are defined as
 \[v_{2p-1}(V):=M_\Phi(v_{2p-1})\in H^{2p-1}(G), \ \ 1\leq p\leq n, \]
 where $n$ is the rank of $V$. If $V$ is a real representation then they are defined as $v_{2p-1}(V):=v_{2p-1}(V_\C),$
 where $V_\C$ is the complexification of $V$. 
\end{definition} 

\begin{lemma}\label{lemma:ch-2}
The characteristic classes of the representation $\Phi$ do not depend on the choice of $z\in M$, $F$ as in \eqref{eq:F}, nor the Hermitian metric $h$ of $V$.
\end{lemma}

Before proving Lemma \ref{lemma:ch-2} and Corollary \ref{cor:VE-injective}, let us point out the relation between the characteristic classes of a representation of $G$, and those of the induced representation on its algebroid:

\begin{proposition}\label{prop:gpd-to-algb}
Let $\Psi\colon A\to \mathfrak{gl}(V)$ be the induced representation \eqref{eq:induced} of $\Phi:G\to GL(V)$ on the Lie algebroid $A$ of $G$. Then the characteristic classes of $\Psi$ are the image of the characteristic classes of $\Phi$ via the van Est map $VE_G\colon H(G)\to H(A)$:
\[VE_G(v_{2p-1}(V))=u_{2p-1}(V), \ \ 1\leq p\leq n. \]
\end{proposition}

\begin{proof}
This follows by the fact that $VE_{GL_n}(v_{2p-1})=u_{2p-1}, \ \ 1\leq p\leq n$ (as $VE_{GL_n}\circ R_{GL_n/U(n)}=\id_{\sC^\bullet(\gln)_{U(n)-\mathrm{basic}}}$) and the commutativity of the diagram
\begin{equation}\label{eq:diagram-VE}
\xymatrix{
\widetilde{\sC}(GL_n)^{U(n)} \ar[d]^-{VE_{GL_n}} \ar[r]^-{C_F^*} & \widetilde{\sC}(GL(V_z))^{U(h_z)}  \ar[d]^-{VE_{GL(V_z)}} \ar[r]^-{\mathrm{ext}} & \widetilde{\sC}(GL(V))^{U(h)} \ar[d]^{VE_{GL(V)}} \ar[r]^-{\Phi^*}& \sC(G) \ar[d]^-{VE_G} \\
\sC(\gln)_{U(n)-\mathrm{basic}} \ar[r]^-{C_F^*} & \sC(\mathfrak{gl}(V_z))_{U(h_z)-\mathrm{basic}} \ar[r]^-{\mathrm{ext}}& \sC(\mathfrak{gl}(V))_{U(h)-\mathrm{basic}} \ar[r]^-{\Psi^*}& \sC(A),
}
\end{equation}
where the middle square commutes by Corollary \ref{cor:injective}.
\end{proof}

We now proceed to the proof of the injectivity of the van Est map $VE_{GL(V)}\colon H(GL(V))\to H(\mathfrak{gl}(V))$.

\begin{proof}[Proof of Corollary \ref{cor:VE-injective}] Consider the commutative diagram
\begin{equation}\label{eq:final}
\xymatrix{ \widetilde{\sC}(GL(V))^{U(h)} \ar[d]^{VE_{GL(V)}} \ar[r]^-{\iota_{GL(V)}}& \sC(GL(V))  \ar[d]^-{VE_{GL(V)}}\\
 \sC(\mathfrak{gl}(V))_{U(h)-\mathrm{basic}}  \ar[r]^-\iota & \sC(\mathfrak{gl}(V)) .
}
\end{equation}

The upper map is a homotopy equivalence by Theorem \ref{thm:averaging}.
The left map is also a homotopy equivalence: Set $G:=GL(V)$ and $K:=U(h)$ in Corollary \ref{cor:r-ve}. In the left commutative diagram of Corollary \ref{cor:r-ve} we have that the restriction maps $r$ are always homotopy equivalences and $VE_{GL(V_z)}$ -- the van Est map at the level of the isotropies -- is a homotopy equivalence by Theorem \ref{prop:rightinversegk} (see also Section \ref{section:example}). Hence $VE_{GL(V)}$ is a homotopy equivalence.

On the other hand, $\iota\colon H(\mathfrak{gl}(V)),U(h))\to H(\mathfrak{gl}(V))$ is injective: The inclusion\linebreak $\iota_z\colon H(\mathfrak{gl}(V_z)),U(h_z))\to H(\mathfrak{gl}(V_z))$ is injective (see \eqref{eq:inc_H}), and the restriction map $r\colon\sC(\mathfrak{gl}(V))\to \sC(\mathfrak{gl}(V_z))$ is a homotopy equivalence (see the proof of Theorem \ref{thm:transitive} with the transitive groupoid $G=GL(V)$). Hence, by the commutativity of the diagram
\[
\xymatrix{ 
\sC(\mathfrak{gl}(V_z))_{U(h_z)\mathrm{-basic}}  \ar[r]^-{\iota_z} &   \sC(\mathfrak{gl}(V_z))  \\
\sC(\mathfrak{gl}(V))_{U(h)\mathrm{-basic}} \ar[r]^-\iota \ar[u]_-{r}& \sC(\mathfrak{gl}(V)) \ar[u]_-{r},
}
 \]
$\iota$ is injective at the level of cohomology. Hence, by the commutativity of \eqref{eq:final}, $VE_{GL(V)}\colon H(GL(V))\to H(\mathfrak{gl}(V))$ is injective.
\end{proof}

\begin{proof}[Proof of Lemma \ref{lemma:ch-2}]
The right square of the diagram \eqref{eq:diagram-VE} factors through the inclusions as 
\begin{equation}\label{eq:1}
\xymatrix{ \widetilde{\sC}(GL(V))^{U(h)} \ar[d]^{VE_{GL(V)}} \ar[r]^-{\iota_{GL(V)}}& \sC(GL(V)) \ar[r]^-{\Phi^*} \ar[d]^-{VE_{GL(V)}}& \sC(G) \ar[d]^-{VE_G}\\
 \sC(\mathfrak{gl}(V))_{U(h)-\mathrm{basic}}  \ar[r]^-\iota & \sC(\mathfrak{gl}(V)) \ar[r]^-{\Psi^*}&\sC(A).
}
\end{equation}
To see that the classes do not depend on the choices, recall that $VE_{GL(V)}\colon H(GL(V))\to H(\mathfrak{gl}(V))$ is injective by Corollary \ref{cor:VE-injective}. In this case, the result follows at once by this fact, \eqref{eq:1} and Lemma \ref{prop:ch}. 
\end{proof}

The main properties of the characteristic classes are:

\begin{proposition}\label{prop:properties} For any representations $V,W$ of $G$, the characteristic classes satisfy the following properties:
\begin{enumerate}
\item $v_{2p-1}(V\oplus W)=v_{2p-1}(V)+v_{2p-1}(W).$
\item $v_{2p-1}(V\otimes W)=v_{2p-1}(V)\rk(W)+\rk(V) v_{2p-1}(W).$
\item $v_{2p-1}(V^*)=i^*v_{2p-1}(V),$ where $i:G\to G$ is the inverse map. In particular, \[v_{1}(V^*)=-v_1(V).\]
\item $v_{2p-1}(V)=0$ if $p$ is even and $V$ is real. 
\end{enumerate}
\end{proposition}

\begin{proof}
Property (c) follows directly from the definition of the action on the dual $V^*$. The proofs of (a), (b) and (d) are similar; we will show (b) and (d).\\

For (b), consider the Lie subgroup $GL(V)\otimes GL(W)\subset GL(V\otimes W)$ consisting of the tensor product $\phi\otimes \psi\colon V_m\otimes W_m\to V_{m'}\otimes W_{m'}$ of linear isomorphisms  $\phi\colon V_m\to V_{m'}$, $\psi\colon W_m\to W_{m'}$. Similarly, if $h,h'$ are Hermitian metrics for $V$ and $W$, respectively, the Lie subgroupoid $U(h)\otimes U(h')\subset GL(V)\otimes GL(W)$ is proper and transitive and is a subgroupoid of $U(h\otimes h').$

If $\Phi_V\colon G\to GL(V)$ and $\Phi_W\colon G\to GL(W)$ denote the representations of $G$, then its induced representation on $V\otimes W$ factors as $\Phi_V\otimes \Phi_W\colon G\to GL(V)\otimes GL(W)$, and 
\begin{equation*}
\xymatrix{
\widetilde{\sC}(GL(V\otimes W))^{U(h\otimes h')}\ar[r]^-{\Phi_V^*\otimes \Phi_W^*}  \ar[d]_-{\iota^*} &\sC(G)\\
\widetilde{\sC}(GL(V)\otimes GL(W))^{U(h)\otimes U( h')}\ar[ru]_-{\Phi_V^*\otimes \Phi_W^*} & 
}
\end{equation*}
commutes, where $\iota\colon GL(V)\otimes GL(W)\to GL(V\otimes W)$ is the inclusion.
Denoting by $n=\rk(V), m=\rk(W)$, it is then a straightforward check to see that the composition $$M_{\Phi_V^*\otimes \Phi_W^*}\colon\widetilde{\sC}(GL_{nm})^{U(nm)}\to \sC(G)$$ from \eqref{eq:classes_composition} factors as 
\[
\xymatrix{
 \widetilde{\sC}(GL_{nm})^{U(nm)} \ar[r]^-{\iota^*}  & \widetilde{\sC}(GL_n\otimes GL_m)^{U(n)\otimes U(m)} \ar[r]^-{C_F^*\otimes C_{F'}^*} & \widetilde{\sC}(GL(V_z)\otimes GL(W_z))^{U(h_z)\otimes U(h'_z)}& &
 }
 \]
 \[
 \xymatrix{
& & & & & & \ar[r]^-{\mathrm{ext}\otimes \mathrm{ext}} & \widetilde{\sC}(GL(V)\otimes GL(W))^{U(h)\otimes U(h')} \ar[r]^-{\Phi_V^*\otimes \Phi_W^*} & \sC(G).
}
\]
It is then enough to see that for $v_{2p-1}(nm)\in  \widetilde{\sC}(GL_{nm})^{U(nm)}$ the generators as in \eqref{eq:generators}
$$\iota^*(v_{2p-1(nm)})=v_{2p-1}(n)m+nv_{2p-1}(m),$$
with $v_{2p-1}(n)\in  \widetilde{\sC}(GL_{n})^{U(n)}$, $v_{2p-1}(m)\in  \widetilde{\sC}(GL_{m})^{U(m)}$ the generators, accordingly. 

Indeed, we use the {\it Kronecker product} $A\otimes B\in GL_{nm}$ for matrices $A\in GL_n,B\in GL_m$ which represents the tensor product of the two invertible maps $A, B$. By general properties of the Kronecker product one can verify that  if $[A]=e^X$ modulo $U(n)$ and $[B]=e^Y$ modulo $U(m)$, then $[A\otimes B]=e^X\otimes e^Y$ modulo $U(nm)$ and $\lambda_t(e^X\otimes e^Y)=e^{tX}\otimes e^{tY}$ ($\lambda_t$ as in \eqref{eq:lambda}). Therefore, for $A_1,\ldots, A_q\in GL_n,\ B_1,\ldots, B_q\in GL_m$ 
\[\gamma^{(q)}_{t_1,\ldots, t_q}(A_1\otimes B_1,\ldots A_q\otimes B_q)=\gamma^{(q)}_{t_1,\ldots, t_q}(A_1,\ldots A_q)\otimes \gamma^{(q)}_{t_1,\ldots, t_q}( B_1,\ldots B_q)\]
$\gamma$ as in \eqref{eq:gammaformula0}, and by bilinearity of $\otimes$ and the fact that 
\begin{equation}\label{eq:mult}(A\otimes B)(C\otimes D)=AC\otimes BD,\end{equation} it follows that 
\begin{multline*}
\theta\left(\frac{\partial}{\partial t_i}(\gamma^{(q)}_{t_1,\ldots, t_q}(A_1\otimes B_1,\ldots A_q\otimes B_q))\right)=\\
\theta\left(\frac{\partial}{\partial t_i}(\gamma^{(q)}_{t_1,\ldots, t_q}(A_1,\ldots A_q))\right)\otimes I_m+I_n\otimes\theta\left(\frac{\partial}{\partial t_i}(\gamma^{(q)}_{t_1,\ldots, t_q}(A_1,\ldots A_q))\right)
\end{multline*}
with $\theta$ the left invariant MC-form of $GL_{nm}, GL_n$ and $GL_m$, respectively. Again, by property \eqref{eq:mult} and $\mathrm{tr}(A\otimes B)=\mathrm{tr}(A)\mathrm{tr}(B)$, and recalling that $u_q\in C^q(\mathfrak{gl}_{nm})_{U(nm)-\mathrm{basic}},\ q=2p-1$ are sums of traces as in \eqref{eq:u}, we further obtain that 
\begin{multline*}
\iota^*(v_q(nm))(A_1\otimes B_1,\ldots A_q\otimes B_q)=\int_{[0,1]^q}\gamma^{(q)}(A_1\otimes B_1,\ldots A_q\otimes B_q)^*\iota^*(u_{q,GL_{nm}/U{nm}})
=\\\mathrm{tr}(I_m)\int_{[0,1]^q}\gamma^{(q)}(A_1,\ldots A_q)^*u_{q,GL_{n}/U{n}}+\mathrm{tr}(I_n)\int_{[0,1]^q}\gamma^{(q)}(B_1,\ldots B_q)^*u_{q,GL_{m}/U{m}}
=\\mv_q(n)(A_1,\ldots A_q)+nv_q(m)(B_1,\ldots B_q).
\end{multline*}

For (d) notice that the obvious inclusion $\gln(\R)\subset \gln(\C)$ factors through

 the diagram
\begin{equation*}
\xymatrix{
\sC(\gln(\C))_{U(n)-\mathrm{basic}}\ar[rr]^-{R_{GL_n(\C)/U(n)}}  \ar[d]_-{\iota^*} & & \widetilde{\sC}(GL_{n}(\C))^{U(n)}\\
\sC(\gln(\R))_{O(n)-\mathrm{basic}} \ar[rru]_-{R_{GL_n(\R)/O(n)}} & &
}
\end{equation*} 
and that the restriction to $\gln(\R)$ of the generators $u_{2q-1}$ of $H (\gln(\C),U(n))$ vanish for $q$ even (see \cite{kamber}).
\end{proof}

 \begin{appendix}
\section{Van Est maps using the Perturbation Lemma}\label{appendix}
 In this appendix we recall some of the constructions and results of \cite{Eckhard} (see also \cite{Eckhard_2}). Let 
$(\sD^{\bullet,\bullet},\d,\delta)$ be a double complex, concentrated in non-negative degrees.
Let $(\sX^{\bullet},\d)$ be a cochain complex. A morphism of double complexes
\[ \si\colon \sX\to  \sD\]
 (where $\sX^\bullet$ is regarded as a double complex concentrated in bidegrees $(0,\bullet)$) will be called a \emph{horizontal augmentation map}. 
Passing to total complexes, $\si$ becomes a  cochain map from $X^\bullet$ to the cochain complex $(\on{Tot}^\bullet(\sD),\d+\delta)$. The Perturbation Lemma, due to Brown \cite{Brown} and Gugenheim \cite{Gugen}, 
allows us to turn a homotopy operator for the horizontal differential $\delta$ into a homotopy operator with respect to the total differential $\d+\delta$. 
%
\begin{lemma}[Perturbation Lemma] \label{lem:perturbed}
		Suppose $\mathsf{h}\colon \sD\to \sD$ is a linear map of bidegree  $(-1,0)$ such that 
		\[ [\mathsf{h},\delta]=1-\si\circ \sp\] for some degree $0$ map 
		$\sp\colon \sD^{0,\bullet}\to \sX^\bullet$.
		Put 
		$\mathsf{h}'=\mathsf{h}(1+\d \mathsf{h})^{-1}$ and $ \sp'=\sp(1+\d \mathsf{h})^{-1}$. 
		Then 
		\[ [\mathsf{h}',\d+\delta]=1-\si\circ\sp'.\]
\end{lemma}
Here, $[\cdot,\cdot]$ denotes the \emph{graded} commutator, e.g. $[\mathsf{h},\delta]=\mathsf{h}\delta+\delta\mathsf{h}$.  
	
We shall assume from now on that $\sp\circ \si=\id_{\sX}$, so that $\si$ is injective and $\si\circ \sp$ is a projection onto the image of $\si$. Then also $\sp'\circ \si=\id_{\sX}$, and $\si\circ \sp'$ is again a projection. 
In other words, $\si\colon \sX^\bullet\to \on{Tot}^\bullet(\sD)$ is a homotopy equivalence, with $\sp'$ a homotopy inverse.

In our applications, there is another cochain complex $(\sY^\bullet,\delta)$, with a
\emph{vertical augmentation map} 
\[ \sj\colon \sY\to \sD\]
 (thus $\sY^\bullet$ is regarded as a double complex concentrated in bidegrees $(\bullet,0)$). The horizontal homotopy $\mathsf{h}$ allows us to `invert' the second cochain map in 
\[ \sY^\bullet\stackrel{\sj}{\lra} \on{Tot}^\bullet(\sD)\stackrel{\si}{\longleftarrow}\sX^\bullet,\]
thereby producing a cochain map 
$\sp'\circ \sj=\sp\circ (1+\d \mathsf{h})^{-1}\circ \sj\colon \sY^\bullet\to \sX^\bullet$. 
On elements of degree $p$, this is given by 
\begin{equation}\label{eq:xy} (-1)^p \sp \circ (\d \mathsf{h})^p \circ \sj\colon \sY^p\to \sX^p,\end{equation}

Consider now the situation that the vertical differential has a homotopy operator 
\[ \mathsf{k}\colon \sD^{p,q}\to \sD^{p,q-1},\ \ [\d, \mathsf{k}]=1-\sj\circ \sq,\]
where 
$\sq\colon \sD^{0,\bullet}\to \sY^\bullet$ is a cochain map for $\d$ with $\sq\circ \sj=\id_{\sY}$. Then we can apply the Perturbation Lemma \ref{lem:perturbed} to 
this vertical homotopy, and we obtain a cochain map $\sq\circ (1+\delta \mathsf{k})^{-1}\circ \si\colon \sX^\bullet\to \sY^\bullet$ given on degree $p$ elements by 
\begin{equation}\label{eq:yx}  (-1)^p \sq \circ (\delta \mathsf{k} )^p \circ \si\colon \sX^p\to \sY^p.
\end{equation}
%
The following result will be used to relate van Est `integration' and `differentiation'  maps. 

\begin{lemma}\label{lem:backandforth}
Suppose the homotopy operators $\mathsf{h}, \mathsf{k}$ 
satisfy 
\begin{equation}\label{eq:hk0} \mathsf{h}\circ  \mathsf{k}=0,\ \ \ \ \sp\circ  \mathsf{k}=0.\end{equation}
Then \eqref{eq:yx} followed by 
\eqref{eq:xy} is the identity map of $\sX^p$.
\end{lemma}

For our application we start with a Lie groupoid $G$ with Lie algebroid $A$. The two complexes $\sX=\sC(A),\  \, \sY=\sC(G)$ are related by a double complex, due to Crainic \cite{Crainic:vanEst}. Recall from the discussion around \eqref{eq:kappap} that 
\[\kappa\colon E G\to B G\]
is a simplicial principal bundle, where for each $p$, the $G$-action on $E_pG$ happens along the map $\pi_p\colon E_pG\to M$. The 
double complex 
\begin{equation}\label{eq:double} (\sD^{\bullet,\bullet}(G),\delta,\d)\end{equation} 
is defined as follows.
\begin{itemize}
\item The bigraded summands of the double complex
are 
\begin{equation}\label{eq:2altdef} 
\sD^{p,q}(G)=\Gamma\big(\pi_p^*(\wedge^q A^*)\big).\end{equation}
\item 
$\delta$ is the simplicial differential on sections of the simplicial vector bundle 
\[ \pi_\bullet^*(\wedge^q A^*)\to E_\bullet G.\] 
\item 
$\d=(-1)^p\d_{CE} $ on elements of bidegree $(p,q)$, is the Chevalley-Eilenberg differential on 
\begin{equation}\label{eq:dlacomplex}
\sD^{p,\bullet}(G)\cong \sC^\bullet(\pi_p^*A)\end{equation}
Specifically, let $\F$ be the foliation of $E_pG$ given by the $\kappa_p$-fibers; thus $T_\F E_pG$ is the vertical bundle. The isomorphism 
$\pi_p^* A\cong T_\F E_pG$ defines an algebroid structure on 
$\pi_p^*A$.
\end{itemize}
Furthermore, the double complex comes with horizontal and vertical augmentation maps: 
\begin{itemize}
\item  $\si\colon \sC^\bullet(A)\to 
\sD^{0,\bullet}(G)$ is given in degree $q$ by the pullback $\pi_0^*$ (using \eqref{eq:dlacomplex} for $p=0$). 
%
\item 
$\sj\colon \sC^\bullet(G)
\to \sD^{\bullet,0}(G)$ is given in degree $p$ by the pullback map $\kappa_p^*$ (using \eqref{eq:kappap}).
%
\end{itemize}
The augmentation maps define cochain maps to the total complex
\[ \sC^\bullet(G)\stackrel{\sj}{\lra} \on{Tot}^\bullet(\sD(G))\stackrel{\si}{\longleftarrow}\sC^\bullet(A) .\]

\subsection*{Differentiation}The double complex $(\sD^{\bullet,\bullet}(G),\delta,\d)$ has a horizontal homotopy. Consider the maps
\[ h_p\colon E_pG\to E_{p+1}G,\ (g_1,\ldots,g_p;g)\mapsto (g_1,\ldots,g_p,g;m),\]
where $m=\sz(g)$. Since  $\pi_{p+1}\circ h_{p}=\pi_{p}$, these lift to fiberwise isomorphisms of vector bundles $\pi_{p}^*(\wedge^q A^*)\to \pi_{p+1}^*(\wedge^q A^*)$, defining a  pullback map on sections. 
The map 
\[ \mathsf{h}\colon \sD^{p,q}(G)\to \sD^{p-1,q}(G),\ \ \psi\mapsto (-1)^p h_{p-1}^*\psi\]
satisfies 
$[\mathsf{h},\delta]=1-\si\circ \sp,$
where $\sp\colon \sD^{0,\bullet}(G)\to \sC^\bullet(A)$ is the 
left inverse to $\si=\pi_0^*$
given by pullback under the inclusion  $u\colon M\hookrightarrow E_0 G=G$:
\[ \sp={\uz}^*\colon \Gamma(\pi_0^*(\wedge^q A^*))\to \Gamma(\wedge^q A^*).\] 

The Perturbation Lemma \ref{lem:perturbed} gives a new projection $\sp'=\sp\circ (1+\d \mathsf{h})^{-1}$,  which is a cochain map for the total differential $\d+\delta$, with $\sp'\circ \si=\id$. Thus, $\si$ is a homotopy equivalence, with $\sp'$ a homotopy inverse.  As a result from \cite{Eckhard_2}, we have the following:
\begin{theorem}\label{Thm:appendix}
The van Est differentiation map \eqref{eq:vanEst_expression} is obtained as 
\begin{equation*}
VE_G=	\sp \circ (1+\d \mathsf{h})^{-1}\circ \sj\colon \sC^\bullet(G)\to 
\sC^\bullet(A).
\end{equation*}
\end{theorem}

\end{appendix}

\bibliographystyle{abbrv}
\bibliography{bibliography_serrapilheira}

\end{document}